%% file: NullSpacePropSDP10.tex
\documentclass[12pt]{article}
\usepackage{fullpage,times,graphicx,amssymb,psfrag,color}
\usepackage[round]{natbib}

\input defs.tex

\newtheorem{lemma}{Lemma}

\begin{document}

\title{Testing the Nullspace Property\\ using Semidefinite Programming}
\author{Alexandre d'Aspremont\thanks{In alphabetical order. ORFE Dept., Princeton University, Princeton, NJ 08544. \texttt{aspremon@princeton.edu}} \and Laurent El Ghaoui\thanks{EECS  Dept., U.C. Berkeley, Berkeley, CA 94720. \texttt{elghaoui@eecs.berkeley.edu}} }
\maketitle

\begin{abstract}
Recent results in compressed sensing show that, under certain conditions, the sparsest solution to an underdetermined set of linear equations can be recovered by solving a linear program. These results either rely on computing sparse eigenvalues of the design matrix or on properties of its nullspace. So far, no tractable algorithm is known to test these conditions and most current results rely on asymptotic properties of random matrices. Given a matrix $A$, we use semidefinite relaxation techniques to test the nullspace property on $A$ and show on some numerical examples that these relaxation bounds can prove perfect recovery of sparse solutions with relatively high cardinality.
\end{abstract}


\section{Introduction}
\label{s:intro} 
A recent stream of results in signal processing have focused on producing explicit conditions under which the sparsest solution to an underdetermined linear system can be found by solving a linear program. Given a matrix $A\in\reals^{m \times n}$ with $n>m$
and a vector $v\in\reals^m$, writing $\|x\|_0=\Card(x)$ the number of nonzero coefficients in $x$, this means that the solution of the following (combinatorial) $\ell_0$ minimization problem:
\BEQ\label{eq:l0-prob}
\BA{ll}
\mbox{minimize} & \|x\|_0\\
\mbox{subject to} & Ax=v,\\
\EA
\EEQ
in the variable $x\in\reals^n$, can be found by solving the (convex) $\ell_1$ minimization problem:
\BEQ\label{eq:l1-prob}
\BA{ll}
\mbox{minimize} & \|x\|_1\\
\mbox{subject to} & Ax=v,\\
\EA
\EEQ
in the variable $x\in\reals^n$, which is equivalent to a linear program. 

Based on results by \cite{Vers92} and \cite{Affe92}, \cite{Dono05} show that when the solution $x_0$ of (\ref{eq:l0-prob}) is sparse with $\Card(x_0)=k$ and the coefficients of $A$ are i.i.d. Gaussian, then the solution of the $\ell_1$ problem in (\ref{eq:l1-prob}) will always match that of the $\ell_0$ problem in (\ref{eq:l0-prob}) provided $k$ is below an explicitly computable {\em strong recovery} threshold $k_S$. They also show that if $k$ is below another (larger) {\em weak recovery} threshold $k_W$, then these solutions match with an exponentially small probability of failure.

Universal conditions for strong recovery based on sparse extremal eigenvalues were derived in \cite{Cand05} and \cite{Cand06} who also proved that certain (mostly random) matrix classes satisfied these conditions with an exponentially small probability of failure. Simpler, weaker conditions which can be traced back to \cite{Dono01}, \cite{Zhan05} or \cite{Cohe06} for example, are based on properties of the nullspace of $A$. In particular, if we define
\[
\alpha_k=\max_{\{Ax=0,~\|x\|_1=1\}} ~ \max_{\{\|y\|_\infty=1,~\|y\|_1\leq k\}} y^Tx,
\]
these references show that $\alpha_k<1/2$ guarantees strong recovery. 

One key issue with the current sparse recovery conditions in \cite{Cand05} or \cite{Dono01} is that except for explicit recovery thresholds available for certain types of random matrices, testing these conditions on generic matrices is potentially {\em harder} than solving the combinatorial $\ell_0$-norm minimization problem in (\ref{eq:l0-prob}) for example as it implies either solving a combinatorial problem to compute $\alpha_k$, or computing sparse eigenvalues. Semidefinite relaxation bounds on sparse eigenvalues were used in \cite{dAsp08b} or \cite{Lee08} for example to test the restricted isometry conditions in \cite{Cand05} on arbitrary matrices. In recent independent results, \cite{Judi08} provide an alternative proof of some of the results in \cite{Dono01}, extend them to the noisy case and produce a linear programming (LP) relaxation bound on $\alpha_k$ with explicit performance bounds.

In this paper, we derive a semidefinite relaxation bound on $\alpha_k$, study its tightness and performance. By randomization, the semidefinite relaxation also produces lower bounds on the objective value as a natural by-product of the solution. Overall, our bounds are slightly better than LP ones numerically but both relaxations share the same asymptotic performance limits. However, because it involves solving a semidefinite program, the complexity of the semidefinite relaxation derived here is significantly higher than that of the LP relaxation. 

The paper is organized as follows. In Section \ref{s:recov}, we briefly recall some key results in \cite{Dono01} and \cite{Cohe06}. We derive a semidefinite relaxation bound on $\alpha_k$ in Section \ref{s:relax}, and study its tightness and performance in Section \ref{s:tight}. Section \ref{s:algos} describes a first-order algorithm to solve the resulting semidefinite program. Finally, we test the numerical performance of this relaxation in Section \ref{s:numres}.

\paragraph{Notation}
To simplify notation here, for a matrix $X\in\reals^{m \times n}$, we write its columns $X_i$, $\|X\|_1$ the sum of absolute values of its coefficients (not the $\ell_1$ norm of its spectrum) and $\|X\|_\infty$ the largest coefficient magnitude. More classically, $\|X\|_F$ and $\|X\|_2$ are the Frobenius and spectral norms.

\section{Sparse recovery \& the null space property}
\label{s:recov} Given a \emph{coding} matrix $A\in\reals^{m \times n}$ with $n>m$, a \emph{sparse} signal $x_0\in\reals^n$ and an information vector $v\in\reals^m$ such that
\[
v=Ax_0,
\] 
we focus on the problem of perfectly recovering the signal $x_0$ from the vector $v$, assuming the signal $x_0$ is sparse enough. We define the decoder $\Delta_1(v)$ as a mapping from $\reals^m \rightarrow \reals^n$, with
\BEQ\label{eq:delta-one}
\Delta_1(v) ~\triangleq ~ \argmin_{\{x\in\scriptsize{\reals^n}:~Ax=v\}} \|x\|_1.
\EEQ
This particular decoder is equivalent to a linear program which can be solved efficiently. Suppose that the original signal $x_0$ is sparse, a natural question to ask is then: When does this decoder perfectly recover a sparse signal $x_0$? Recent results by \cite{Cand05}, \cite{Dono05} and \cite{Cohe06} provide a somewhat tight answer. In particular, as in \cite{Cohe06}, for a given coding matrix $A\in\reals^{m \times n}$ and $k>0$, we can quantify the $\ell_1$ error of a decoder $\Delta(v)$ by computing the smallest constant $C>0$ such that
\BEQ\label{eq:decoding-error}
\|x-\Delta(Ax)\|_1\leq C \sigma_k(x)
\EEQ
for all $x\in\reals^n$, where 
\[
\sigma_k(x)\triangleq \min_{\{z\in\scriptsize{\reals^n}:~\Card(z)=k\}}\|x-z\|_1
\]
is the $\ell_1$ error of the best $k$-term approximation of the signal $x$ and can simply be computed as the $\ell_1$ norm of the $n-k$ smallest coefficients of $x\in\reals^n$. We now define the \emph{nullspace property} as in \cite{Dono01} or \cite{Cohe06}.
\begin{definition}
A matrix $A\in\reals^{m \times n}$  satisfies the null space property in $\ell_1$ of order $k$ with constant~$C_k$ if and only if 
\BEQ\label{def:nullspace}
\|z\|_1 \leq C_k \|z_{T^c}\|_1
\EEQ
holds for all $z\in\reals^n$ with $Az=0$ and index subsets $T\subset[1,n]$ of cardinality $\Card(T)\leq k$, where $T^c$ is the complement of $T$ in $[1,n]$.
\end{definition}
\cite{Cohe06} for example show the following theorem linking the optimal decoding quality on sparse signals and the nullspace property constant $C_k$.
\begin{theorem}
Given a coding matrix $A\in\reals^{m \times n}$ and a sparsity target $k>0$. If $A$ has the nullspace property in (\ref{def:nullspace}) of order $2k$ with constant $C/2$, then there exists a decoder $\Delta_0$ which satisfies (\ref{eq:decoding-error}) with constant $C$. Conversely, if (\ref{eq:decoding-error}) holds with constant $C$ then $A$ has the nullspace property at the order $2k$ with constant $C$.
\end{theorem}
\begin{proof}
See \cite[Corollary 3.3]{Cohe06}.
\end{proof}

This last result means that the existence of an optimal decoder staisfying (\ref{eq:decoding-error}) is equivalent to $A$ satisfying (\ref{def:nullspace}). Unfortunately, this optimal decoder $\Delta_0(v)$ is defined as
\[
\Delta_0(v)\triangleq \argmin_{\{z\in\scriptsize{\reals^n}:~Az=v\}} \sigma_k(z)
\]
hence requires solving a combinatorial problem which is potentially intractable. However, using tighter restrictions on the nullspace property constant $C_k$, we get the following result about the linear programming decoder $\Delta_1(v)$ in (\ref{eq:delta-one}).

\begin{theorem}\label{th:lp-recov}
Given a coding matrix $A\in\reals^{m \times n}$ and a sparsity target $k>0$. If $A$ has the nullspace property in (\ref{def:nullspace}) of order $k$ with constant $C<2$, then the linear programming decoder $\Delta_1(y)$ in~(\ref{eq:delta-one}) satisfies the error bounds in (\ref{eq:decoding-error}) with constant $
{2C}/{(2-C)}$ at the order $k$.
\end{theorem}
\begin{proof}
See steps (4.3) to (4.10) in the proof of \cite[Theorem 4.3]{Cohe06}.
\end{proof}

To summarize the results above, if there exists a $C>0$ such that the coding matrix $A$ satisfies the nullspace property in (\ref{def:nullspace}) at the order $k$ then there exists a decoder which perfectly recovers signals $x_0$ with cardinality $k/2$. If, in addition, we can show that $C<2$, then the linear programming based decoder in (\ref{eq:delta-one}) perfectly recovers signals $x_0$ with cardinality $k$. In the next section, we produce upper bounds on the constant $C_k$ in (\ref{def:nullspace}) using semidefinite relaxation techniques.

\section{Semidefinite Relaxation}
\label{s:relax} Given $A\in\reals^{m \times n}$ and $k>0$, we look for a constant $C_k \geq 1$ in (\ref{def:nullspace}) such that
\[
\|x_T\|_1 \leq (C_k-1) \|x_{T^c}\|_1
\]
for all vectors $x\in\reals^n$ with $Ax=0$ and index subsets $T\subset [1,n]$ with cardinality $k$. We can rewrite this inequality
\BEQ\label{eq:ineq-C}
\|x_T\|_1 \leq \alpha_k \|x\|_1
\EEQ
with $\alpha_k \in [0,1)$. Because $\alpha_k=1-1/C_k$, if we can show that $\alpha_k<1$ then we prove that $A$ satisfies the nullspace property at order $k$ with constant $C_k$. Furthermore, if we prove $\alpha_k<1/2$, we prove the existence of a linear programming based decoder which perfectly recovers signals $x_0$ with at most $k$ errors. By homogeneity, the constant $\alpha_k$ can be computed as
\BEQ\label{eq:max-max-pb}
\alpha_k=\max_{\{Ax=0,~\|x\|_1=1\}} ~ \max_{\{\|y\|_\infty=1,~\|y\|_1\leq k\}} y^Tx,
\EEQ
where the equality $\|x\|_1=1$ can, without loss of generality, be replaced by $\|x\|_1\leq 1$. We now derive a semidefinite relaxation for problem (\ref{eq:max-max-pb}) as follows. After a change of variables
\[
\left(\BA{cc}
X & Z^T\\
Z & Y\\
\EA\right)
=
\left(\BA{cc}
xx^T & xy^T\\
yx^T & yy^T\\
\EA\right),
\]
we can rewrite (\ref{eq:max-max-pb}) as
\BEQ\label{eq:rank-one-prob}
\BA{ll}
\mbox{maximize} & \Tr(Z)\\
\mbox{subject to} & AXA^T=0,~\|X\|_1\leq 1,\\
& \|Y\|_\infty\leq1,~\|Y\|_1\leq k^2,~\|Z\|_1\leq k,\\
& \left(\BA{cc}
X & Z^T\\
Z & Y\\
\EA\right)\succeq 0,~
\Rank \left(\BA{cc}
X & Z^T\\
Z & Y\\
\EA\right)=1,\\
\EA\EEQ
in the variables $X,Y\in\symm_n$, $Z\in\reals^{n \times n}$, where all norms should be understood componentwise. We then simply drop the rank constraint to form a relaxation of (\ref{eq:max-max-pb}) as
\BEQ\label{eq:max-max-relax}
\BA{ll}
\mbox{maximize} & \Tr(Z)\\
\mbox{subject to} & AXA^T=0,~\|X\|_1\leq 1,\\
& \|Y\|_\infty\leq1,~\|Y\|_1\leq k^2,~\|Z\|_1\leq k,\\
& \left(\BA{cc}
X & Z^T\\
Z & Y\\
\EA\right)\succeq 0,
\EA
\EEQ
which is a semidefinite program in the variables $X,Y\in\symm_n$, $Z\in\reals^{n \times n}$. Note that the contraint $~\|Z\|_1\leq k$ is redundant in the rank one problem but not in its relaxation. Because all constraints are linear here, dropping the rank constraint is equivalent to computing a Lagrangian (bidual) relaxation of the original problem and adding redundant constraints to the original problem often tightens these relaxations. The dual of program (\ref{eq:max-max-relax}) can be written
\[
\BA{ll}
\mbox{minimize} & \|U_1\|_\infty + k^2 \|U_2\|_\infty + \|U_3\|_1+k\|U_4\|_\infty\\
\mbox{subject to} 
& \left(\BA{cc}
U_1-A^TWA & -\frac{1}{2} ({\idm} + U_4)\\
-\frac{1}{2} ({\idm} + U_4^T) & U_2+U_3\\
\EA\right)\succeq 0,
\EA
\]
which is a semidefinite program in the variables $U_1,U_2,U_3,W\in\symm_n$ and $U_4\in\reals^{n \times n}$. For any feasible point of this program, the objective $\|U_1\|_\infty + k^2 \|U_2\|_\infty + \|U_3\|_1+k\|U_4\|_\infty$ is an upper bound on the optimal value of (\ref{eq:max-max-relax}), hence on~$\alpha_k$. We can further simplify this program using elimination results for LMIs. In fact, \cite[\S2.6.2]{Boyd94} shows that this last problem is equivalent to
\BEQ\label{eq:sdp-kernel}
\BA{ll}
\mbox{minimize} & \|U_1\|_\infty + k^2 \|U_2\|_\infty + \|U_3\|_1+k\|U_4\|_\infty\\
\mbox{subject to} 
& \left(\BA{cc}
U_1-wA^TA &-\frac{1}{2} ({\idm} + U_4)\\
-\frac{1}{2} ({\idm} + U_4^T) & U_2+U_3\\
\EA\right)\succeq 0,
\EA
\EEQ
where the variable $w$ is now scalar. In fact, using the same argument, letting $P\in\reals^{n\times p}$ be an orthogonal basis of the nullspace of $A$, i.e. such that $AP=0$ with $P^TP=\idm$, we can rewrite the previous problem as follows
\BEQ\label{eq:dual-ker}
\BA{ll}
\mbox{minimize} & \|U_1\|_\infty + k^2 \|U_2\|_\infty + \|U_3\|_1+k\|U_4\|_\infty\\
\mbox{subject to} 
& \left(\BA{cc}
P^T U_1 P & -\frac{1}{2} P^T ({\idm} + U_4)\\
-\frac{1}{2}  ({\idm} + U_4^T)P & U_2+U_3\\
\EA\right)\succeq 0,
\EA
\EEQ
which is a (smaller) semidefinite program in the variables $U_1,U_2,U_3\in\symm_n$ and $U_4\in\reals^{n \times n}$. The dual of this last problem is then
\BEQ\label{eq:primal-ker}
\BA{ll}
\mbox{maximize} & \Tr(Q_2^TP) \\
\mbox{subject to} & \|PQ_1P^T\|_1\leq 1,~\|PQ_2^T\|_1\leq k\\
& \|Q_3\|_\infty\leq1,~\|Q_3\|_1\leq k^2\\
& \left(\BA{cc}
Q_1 & Q_2^T\\
Q_2 & Q_3\\
\EA\right)\succeq 0,
\EA\EEQ
which is a semidefinite program in the matrix variables $Q_1\in\symm_p,~Q_2\in\reals^{p \times n},~Q_3\in\symm_n$, whose objective value is equal to that of problem (\ref{eq:max-max-relax}). 

Note that adding any number of redundant constraints in the original problem (\ref{eq:rank-one-prob}) will further improve tightness of the semidefinite relaxation, at the cost of increased complexity. In particular, we can use the fact that when
\[
\|x\|_1=1,~\|y\|_\infty=1,~\|y\|_1\leq k,
\]
and if we set $Y=yy^T$ and $Z=yx^T$, we must have
\[
\sum_{i=1}^n |Y_{ij}| \leq k t_j,~|Y_{ij}|\leq t_j,~\ones^Tt\leq k,~t \leq \ones,\quad\mbox{for }i,j=1,\ldots,n,
\]
and
\[
\sum_{i=1}^n |Z_{ij}| \leq k r_j,~|Z_{ij}|\leq r_j,~\ones^T r\leq k,\quad\mbox{for }i,j=1,\ldots,n,
\]
for $r,t\in\reals^n$. This means that we can refine the constraint $\|Z\|_1\leq k$ in (\ref{eq:max-max-relax}) to solve instead
\BEQ\label{eq:max-max-relax-col}
\BA{ll}
\mbox{maximize} & \Tr(Z)\\
\mbox{subject to} & AXA^T=0,~\|X\|_1\leq 1,\\
& \sum_{i=1}^n |Y_{ij}| \leq k t_j,~|Y_{ij}|\leq t_j,~\ones^Tt\leq k,~t \leq \ones,\\
& \sum_{i=1}^n |Z_{ij}| \leq k r_j,~|Z_{ij}|\leq r_j,~\ones^T r\leq 1,\quad\mbox{for }i,j=1,\ldots,n,\\
& \left(\BA{cc}
X & Z^T\\
Z & Y\\
\EA\right)\succeq 0,
\EA
\EEQ
which is a semidefinite program in the variables $X,Y\in\symm_n$, $Z\in\reals^{n \times n}$ and $r,t\in\reals^n$. Adding these columnwise constraints on $Y$ and $Z$ significantly tightens the relaxation. Any {\em feasible} solution to the dual of (\ref{eq:max-max-relax-col}) with objective value less than $1/2$ will then be a certificate that $\alpha_k<1/2$.

\section{Tightness \& Limits of Performance}
\label{s:tight} The relaxation above naturally produces a covariance matrix as its output and we use randomization techniques as in \cite{Goem95} to produce primal solutions for problem (\ref{eq:max-max-pb}). Then, following results by A. Nemirovski (private communication), we bound the performance of the relaxation in~(\ref{eq:max-max-relax}). 

\subsection{Randomization}
Here, we show that lower bounds on $\alpha_k$ can be generated as a natural by-product of the relaxation. We use solutions to the semidefinite program in (\ref{eq:max-max-relax}) and generate feasible points to (\ref{eq:max-max-pb}) by randomization. These can then be used to certify that $\alpha_k>1/2$ and prove that a matrix does not satisfy the nullspace property. Suppose that the matrix
\BEQ\label{eq:gamma}
\Gamma=\left(\BA{cc}
X & Z^T\\
Z & Y\\
\EA\right)
\EEQ
solves problem (\ref{eq:max-max-relax}), because $\Gamma\succeq 0$, we can generate Gaussian variables $(x,y)\sim \mathcal{N}(0,\Gamma)$. Below, we show that after proper scaling, $(x,y)$ will satisfy the constraints of problem~(\ref{eq:max-max-pb}) with high probability, and use this result to quantify the quality of these randomized solutions. We begin by recalling classical results on the moments of $\|x\|_1$ and $\|x\|_\infty$ when $x\sim\mathcal{N}(0,X)$ and bound deviations above their means using concentration inequalities on Lipschitz functions of Gaussian variables.
\begin{lemma}\label{lem:dev-l1}
Let $X\in\symm_n$, $x\sim\mathcal{N}(0,X)$ and $\delta>0$, we have
\BEQ
\mathbf{P}\left(\frac{\|x\|_1}{(\sqrt{2/\pi}+\sqrt{2\log \delta})\sum_{i=1}^n \left( X_{ii} \right)^{1/2}}\geq 1\right)\leq \frac{1}{\delta}
\EEQ
\end{lemma}
\begin{proof}
Let $P$ be the square root of $X$ and $u_i\sim\mathcal{N}(0,1)$ be independent Gaussian variables, we have
\[
\|x\|_1=\sum_{i=1}^n \left| \sum_{j=1}^n P_{ij} u_j \right|
\]
hence, because each term $|\sum_{j=1}^n P_{ij} u_j |$ is a Lipschitz continuous function of the variables $u$ with constant $(\sum_{j=1}^n P_{ij}^2)^{1/2}=(X_{ii})^{1/2}$, $\|x\|_1$ is Lipschitz with constant $L=\sum_{i=1}^n \left( X_{ii} \right)^{1/2}$. Using the concentration inequality by \cite{Ibra76} (see also \citet{Mass07} for a general discussion) we get for any $\beta>0$
\[
\mathbf{P}\left(\frac{\|x\|_1}{\beta}\geq \frac{\Expect[\|x\|_1]+t}{\beta} \right)\leq \exp\left(-\frac{t^2}{2L^2}\right)
\]
with $\Expect[\|x\|_1]=\sqrt{2/\pi}\sum_{i=1}^n \left( X_{ii} \right)^{1/2}$. Picking $t=\sqrt{2\log \delta}L$ and $\beta=\Expect[\|x\|_1]+t$ yields the desired result.
\end{proof}

We now recall another classic result on the concentration of $\|y\|_\infty$, also based on the fact that $\|y\|_\infty$ is a Lipschitz continuous function of independent Gaussian variables.

\begin{lemma}\label{lem:dev-linf}
Let $Y\in\symm_n$, $y\sim\mathcal{N}(0,Y)$ and $\delta>0$ then
\BEQ
\mathbf{P}\left(\frac{\|y\|_\infty}{(\sqrt{2\log 2n}+\sqrt{2\log \delta })\max_{i=1,\ldots,n} (Y_{ii})^{1/2}} \geq 1\right) \leq \frac{1}{\delta}
\EEQ
\end{lemma}
\begin{proof}
\citep[Theorem 3.12]{Mass07} shows that $\|y\|_\infty$ is a Lipschitz function of independent Gaussian random variables with constant $\max_{i=1,\ldots,n} (Y_{ii})^{1/2}$, hence a reasoning similar to that in lemma \ref{lem:dev-l1} yields the desired result.
\end{proof}

Using union bounds, the lemmas above show that if we pick $3/\delta <1$ and $(x,y)\sim \mathcal{N}(0,\Gamma)$, the scaled sample points
\[
\left(\frac{x}{g(X,\delta)},\frac{y}{h(Y,n,k,\delta)}\right)
\] 
will be feasible in~(\ref{eq:max-max-pb}) with probability at least $1-3/\delta$ if we set
\BEQ \label{eq:defg}
g(X,\delta)=(\sqrt{2/\pi}+\sqrt{2\log \delta})\sum_{i=1}^n \left( X_{ii} \right)^{1/2}
\EEQ
and
\BEQ \label{eq:defh}
h(Y,n,k,\delta)=\max\left\{(\sqrt{2\log 2n}+\sqrt{2\log \delta})\max_{i=1,\ldots,n} (Y_{ii})^{1/2}, \frac{(\sqrt{2/\pi}+\sqrt{2\log \delta})\sum_{i=1}^n \left( Y_{ii} \right)^{1/2}}{k} \right\}
\EEQ
The randomization technique is then guaranteed to produce a feasible point of~(\ref{eq:max-max-pb}) with objective value
\[
\frac{q_{\{1-3/\delta\}}}{g(X,\delta)h(Y,n,k,\delta)}
\]
where $q_{\{1-3/\delta\}}$ is the $1-3/\delta$ quantile of $x^Ty$ when $(x,y)\sim \mathcal{N}(0,\Gamma)$. We now compute a (relatively coarse) lower bound on the value of that quantile.

\begin{lemma}\label{lem:dev-xy}
Let $\epsilon,~\delta>3$ and $(x,y)\sim \mathcal{N}(0,\Gamma)$, with $\Gamma$ defined as in (\ref{eq:gamma}), then
\BEQ\label{eq:dev-xy}
\mathbf{P}\left( \sum_{i=1}^n x_iy_i \geq \Tr(Z) -\frac{\sqrt{3}}{\sqrt{\delta-3}}\sigma \right)\geq \frac{3}{\delta}
\EEQ
where
\[
\sigma^2=\|Z\|_F^2+\Tr(XY).
\]
\end{lemma}
\begin{proof}
Let $S\in\reals^{2n \times 2n}$ be such that $\Gamma=S^TS$ and $(x,y)\sim \mathcal{N}(0,\Gamma)$, we have
\[
\Expect\left[\left(y^Tx\right)^2\right] = \sum_{i,j=1}^n \Expect\left[ (S_i^Tw)(S_{n+i}^Tw) (S_j^Tw)(S_{n+j}^Tw)\right]
\]
where $w$ is a standard normal vector of dimension $2n$. Wick's formula implies
\BEAS
\Expect\left[ (S_i^Tw)(S_{n+i}^Tw) (S_j^Tw)(S_{n+j}^Tw)\right]
& = &\mathop{\bf Haf} \left(\BA{cccc}
X_{ii} & Z_{ii} & X_{ij} & Z_{ij}\\
Z_{ii} & Y_{ii} & Z_{ij} & Y_{ij} \\
X_{ij} & Z_{ij} & X_{jj} & Z_{jj} \\
Z_{ij} & Y_{ij} & Z_{jj} & Y_{jj} \\
\EA\right)\\
& =& Z_{ii}Z_{jj} + Z_{ij}^2 + X_{ij}Y_{ij},
\EEAS
where $\mathop{\bf Haf}(X)$ is the Hafnian of the matrix $X$ (see \cite{Barv07} for example), which means
\[
\Expect\left[(y^Tx)^2\right] = (\Tr(Z))^2+\|Z\|_F^2 +\Tr(XY).
\]
Because $\Expect[y^Tx]=\Expect[\Tr(xy^T)]=\Tr(\Expect[xy^T])=\Tr(Z)$, we then conclude using Cantelli's inequality, which gives
\[
\mathbf{P}\left( \sum_{i=1}^n x_iy_i \leq \Tr(Z) -t\sigma \right)\leq \frac{1}{1+t^2}
\]
having set $t=\sqrt{3}/\sqrt{\delta-3}$.
\end{proof}

We can now combine these results to produce a lower bound on the objective value achieved by randomization. 

\begin{theorem}\label{th:tightness}
Given $A\in\reals^{m \times n}$, $\epsilon>0$ and $k>0$, writing $SDP_k$ the optimal value of (\ref{eq:max-max-relax}), we have
\BEQ\label{eq:quality}
\frac{SDP_k-\epsilon}{g(X,\delta)h(Y,n,k,\delta)} \leq \alpha_k \leq SDP_k
\EEQ
where 
\[
\delta=3+\frac{3(\|Z\|_F^2+\Tr(XY))}{\epsilon^2}.
\]
\[
g(X,\delta)=(\sqrt{2/\pi}+\sqrt{2\log \delta})\sum_{i=1}^n \left( X_{ii} \right)^{1/2}
\]
and
\[
h(Y,n,k,\delta)=\max\left\{(\sqrt{2\log 2n}+\sqrt{2\log \delta})\max_{i=1,\ldots,n} (Y_{ii})^{1/2}, \frac{(\sqrt{2/\pi}+\sqrt{2\log \delta})\sum_{i=1}^n \left( Y_{ii} \right)^{1/2}}{k} \right\}
\]
\end{theorem}
\begin{proof}
If $\Gamma$ solves~(\ref{eq:max-max-relax}) and the vectors $(x,y)$ are sampled according to $(x,y)\sim \mathcal{N}(0,\Gamma)$, then 
\[
\Expect[(Ax)(Ax)^T]=\Expect[Axx^TA^T]=AXA^T=0,
\]
means that we always have $Ax=0$. When $\delta>3$, Lemmas~\ref{lem:dev-l1} and~\ref{lem:dev-linf} show that
\[
\left(\frac{x}{g(X,\delta)},\frac{y}{h(Y,n,k,\delta)}\right)
\] 
will be feasible in~(\ref{eq:max-max-pb}) with probability at least $1-3/\delta$, hence we can get a feasible point for~(\ref{eq:max-max-pb}) by sampling enough variables $(x,y)$. Lemma \ref{lem:dev-xy} shows that if we set $\delta$ as above, the randomization procedure is guaranteed to reach an objective value $y^Tx$ at least equal to
\[
\frac{\Tr(Z)-\epsilon}{g(X,\delta)h(Y,n,k,\delta)}
\]
which is the desired result.
\end{proof}

Note that because $\Gamma \succeq 0$, we have $Z_{ij}^2\leq X_{ii}Y_{jj}$, hence $\|Z\|_F^2\leq \Tr(X)\Tr(Y)\leq k^2$. We also have $\Tr(XY)\leq \|X\|_1 \|Y\|_1 \leq k^2$ hence
\[
\delta \leq 3+\frac{6k^2}{\epsilon^2}.
\]
and the only a priori unknown terms controlling tightness are $\sum_{i=1}^n (X_{ii})^{1/2}$, $\sum_{i=1}^n (Y_{ii})^{1/2}$ and $\max_{i=1,\ldots,n} (Y_{ii})^{1/2}$. Unfortunately, while the third term is bounded by one, the first two can become quite large, with trivial bounds giving
\[
\sum_{i=1}^n (X_{ii})^{1/2}\leq \sqrt{n} \quad \mbox{and} \quad \sum_{i=1}^n (Y_{ii})^{1/2}\leq \sqrt{n},
\]
which means that, in the worst case, our lower bound will be off by a factor $1/n$. However, we will observe in Section \ref{s:numres} that, when $k=1$, these terms are sometimes much lower than what the worst-case bounds seem to indicate. The expression for the tightness coefficient $\gamma$ in~(\ref{eq:gamma}) also highlights the importance of the constraint $\|Z\|_1\leq k$. Indeed, the positive semidefinitess of $2\times 2$ principal submatrices means that $Z_{ij}^2\leq X_{ii}Y_{jj}$, hence
\[
\|Z\|_1 \leq \left(\sum_{i=1}^n (X_{ii})^{1/2}\right)\left(\sum_{i=1}^n (Y_{ii})^{1/2}\right),
\]
so controlling $\|Z\|_1$ potentially tightens the relaxation. This is confirmed in numerical experiments: the relaxation including the (initially) redundant norm constraint on $Z$ is significantly tighter on most examples. Finally, note that better lower bounds on $\alpha_k$ can be obtained (numerically) by sampling $\|x_T\|_1/\|x\|_1$ in (\ref{eq:ineq-C}) directly, or as suggested by one of the referees, solving
\[\BA{ll}
\mbox{maximize} & c^Tx\\
\mbox{subject to} & Ax=0,~ \|x\|_1\leq 1,
\EA\]
in $x\in\reals^n$ for various random vectors $c\in\{-1,0,1\}^n$ with at most $k$ nonzero coefficients. In both cases unfortunately, the moments cannot be computed explicitly so studying performance is much harder.

\subsection{Performance}
Following results by A. Nemirovski (private communication), we can derive precise bounds on the performance of the relaxation in~(\ref{eq:max-max-relax}). 
\begin{lemma}
Suppose $(X,Y,Z)$ solve the semidefinite program in~(\ref{eq:max-max-relax}), then
\[
\Tr(Z)=\alpha_1
\]
and the relaxation is tight for $k=1$.
\end{lemma}
\begin{proof}
First, notice that when the matrices $(X,Y,Z)$ solve (\ref{eq:max-max-relax}), $AX=0$ with 
\[
\left(\BA{cc}
X & Z^T\\
Z & Y\\
\EA\right)\succeq 0
\]
means that the rows of $Z$ also belong to the nullspace of $A$. If $A$ satisfies the nullspace property in (\ref{eq:ineq-C}), we must have $|Z_{ii}|\leq \alpha_1 \sum_{j=1}^n |Z_{ij}|$ for $i=1,\ldots,n,$ hence $\Tr(Z)\leq \alpha_1 \|Z\|_1 \leq \alpha_1$. By construction, we always have $\Tr(Z) \geq \alpha_1$ hence $\Tr(Z)=\alpha_1$ when $Z$ solves~(\ref{eq:max-max-relax}) with $k=1$.
\end{proof}

As in \cite{Judi08}, this also means that if a matrix $A$ satisfies the restricted isometry property at cardinality $O(m)$ (as Gaussian matrices do for example), then the relaxation in~(\ref{eq:max-max-relax}) will certify $\alpha_k<1/2$ for $k=O(\sqrt{m})$. Unfortunately, the results that follow show that this is the best we can hope for here.

Without loss of generality, we can assume that $n=2m$ (if $n\geq 2m$, the problem is harder). Let $Q$ be an orthoprojector on a $(n-m)$-dimensional subspace of the nullspace of $A$, with $\Rank(Q)=n-m=m$. By construction, $\|Q\|_1\leq n \|Q\|_2 = n \sqrt{m}$, $0\preceq Q \preceq \idm$ and of course $AQ=0$. We can use this matrix to construct a feasible solution to problem~(\ref{eq:max-max-relax-col}) when $k=\sqrt{n}$. We set $X=Q/(n\sqrt{m})$, $Y=Q/\sqrt{n}$, $Z=Q/n$, $t_j=1/\sqrt{n}$ and $r_j=1/n$ for $j=1,\ldots,n$. We then have
\[
\|Y_i\|_1=\frac{\|Q_i\|_1}{\sqrt{n}} \leq {\|Q_i\|_2}\leq 1 \leq k t_i, \quad i=1,\ldots,n,
\]
and $\|Y_i\|_\infty\leq \|Y_i\|_2 \leq 1/\sqrt{n}$ with $\ones^Tt\leq k$. We also get
\[
\|Z_i\|_1=\frac{\|Q_i\|_1}{n} \leq \frac{\|Q_i\|_2}{\sqrt{n}} \leq k r_i, \quad i=1,\ldots,n.
\]
With
\[
\left(\BA{cc}
n^{-1}m^{-1/2} & n^{-1}\\
n^{-1} & n^{-1/2}\\
\EA\right)\succeq 0,
\]
the matrices we have defined above form a {\em feasible point} of problem~(\ref{eq:max-max-relax-col}). Because, $\Tr(Z)=\Tr(Q)/n=1/2$, this feasible point proves that the optimal value of~(\ref{eq:max-max-relax-col}) is larger than $1/2$ when $n=2m$ and $k=\sqrt{n}$. This means that the relaxation in~(\ref{eq:max-max-relax-col}) can prove that a matrix satisfies the nullspace property for cardinalities at most $k=O(\sqrt{n})$ and this performance bound is tight since we have shown that it achieves this rate of $O(\sqrt{n})$ for good matrices.

This counter example also produces bounds on the performance of another relaxation for testing sparse recovery. In fact, if we set $X=Q/m$ with $Q$ defined as above, we have $\Tr(X)=1$ with $X\succeq 0$ and 
\[
\|X\|_1=\frac{\|Q\|_1}{m}\leq 2 \sqrt{m}
\]
and $X$ is an optimal solution of the problem
\[\BA{ll}
\mbox{minimize} & \Tr(XAA^T)\\
\mbox{subject to} & \|X\|_1 \leq 2 \sqrt{2m}\\
& \Tr(X)=1,~X\succeq 0,
\EA\]
which is a semidefinite relaxation used in \cite{dAsp04a} and \cite{dAsp08b} to bound the restricted isometry constant $\delta_k(A)$. Because $\Tr(XAA^T)=0$ by construction, we know that this last relaxation will fail to show $\delta_k(A)<1$ whenever $k=O(\sqrt{m})$. Somewhat strikingly, this means that the three different tractable tests for sparse recovery conditions, derived in \cite{dAsp08b}, \cite{Judi08} and this paper, are {\em all} limited to showing recovery at the (suboptimal) rate $k=O(\sqrt{m})$. 

\section{Algorithms}
\label{s:algos} Small instances of the semidefinite program in (\ref{eq:dual-ker}) and be solved efficiently using solvers such as SEDUMI \citep{Stur99} or SDPT3 \citep{Toh96}. For larger instances, it is more advantageous to solve (\ref{eq:dual-ker}) using first order techniques, given a fixed target for $\alpha$. We set $P\in\reals^{n\times p}$ to be an orthogonal basis of the nullspace of  the matrix $A$ in (\ref{eq:ineq-C}), i.e. such that $AP=0$ with $P^TP=\idm$. We also let $\bar \alpha$ be a target critical value for $\alpha$ (such as $1/2$ for example), and solve the following problem
\BEQ\label{eq:dual-binary}
\BA{ll}
\mbox{maximize} & \lambdamin \left(\BA{cc}
P^T U_1 P & -\frac{1}{2} P^T ({\idm} + U_4)\\
-\frac{1}{2}  ({\idm} + U_4^T)P & U_2+U_3\\
\EA\right)\\
\mbox{subject to} & \|U_1\|_\infty + k^2 \|U_2\|_\infty + \|U_3\|_1 + k\|U_4\|_\infty \leq \bar \alpha\\
\EA
\EEQ
in the variables $U_1,U_2,U_3\in\symm_n$ and $U_4\in\reals^{n \times n}$. If the objective value of this last problem is greater than zero, then the optimal value of problem (\ref{eq:dual-ker}) is necessarily smaller than $\bar \alpha$, hence $\alpha \leq \bar \alpha$ in (\ref{eq:max-max-pb}).

Because this problem is a minimum eigenvalue maximization problem over a simple compact (a norm ball in fact), large-scale instances can be solved efficiently using projected gradient algorithms or smooth semidefinite optimization techniques \citep{Nest04a,dAsp04a}. As we show below, the complexity of projecting on this ball is quite low.


\begin{lemma}\label{lem:proj}
The complexity of projecting $(x_0,y_0,z_0,w_0)\in\reals^{3n}$ on
\[
\|x\|_\infty + k^2 \|y\|_\infty + \|z\|_1 + k\|w\|_\infty \leq \alpha
\]
is bounded by $O(n\log n\log_2 (1/\epsilon))$, where $\epsilon$ is the target precision in projecting.
\end{lemma}
\begin{proof}
By duality, solving
\[\BA{ll}
\mbox{minimize} &  \|x-x_0\|^2 + \|y-y_0\|^2 + \|z-z_0\|^2 + \|w-w_0\|^2\\
\mbox{subject to} & \|x\|_\infty + k^2 \|y\|_\infty + \|z\|_1 + k\|w\|_\infty \leq \alpha\\
\EA\]
in the variables $x,y,z \in \reals^n$ is equivalent to solving
\[
\max_{\lambda \geq 0} \min_{x,y,z,w} \|(x,y,z,w)-(x_0,y_0,z_0,w_0)\|^2 + \lambda \|x\|_\infty + \lambda k^2 \|y\|_\infty +   \lambda \|z\|_1 +  \lambda k \|w\|_\infty - \lambda \alpha
\]
in the variable $\lambda\geq 0$. For a fixed $\lambda$, we can get the derivative w.r.t. $\lambda$ by solving four separate penalized least-squares problems. Each of these problems can be solved explicitly in at most $O(n\log n)$ (by shrinking the current point) so the complexity of solving the outer maximization problem up to a precision $\epsilon>0$ by binary search is $O(n\log n\log_2 (1/\epsilon))$
\end{proof}

We can then implement the smooth minimization algorithm detailed in \cite[\S5.3]{Nest03} to a smooth approximation of problem (\ref{eq:dual-binary}) as in \cite{Nest04a} or \cite{dAsp04a} for example. Let $\mu>0$ be a regularization parameter. The function 
\BEQ\label{eq:fmu}
f_\mu(X) = \mu\log\left( \Tr  \exp\left(\frac{X}{\mu}\right)\right)
\EEQ
satifies
\[
\lambdamax(X) \leq f_\mu(X)  \leq \lambdamax(X) + \mu \log n
\]
for any $X\in\symm_n$. Furthermore, $f_\mu(X)$ is a smooth approximation of the function $\lambdamax(X)$, and $\nabla f_\mu(X)$ is Lipschitz continuous with constant $\log n/\mu$. Let $\epsilon>0$ be a given target precision, this means that if we set $\mu=\epsilon/(2 \log n)$ then
\BEQ\label{eq:f-def}
f(U)\equiv -f_\mu
\left(\BA{cc}
-P^T U_1 P & \frac{1}{2} P^T ({\idm} + U_4)\\
\frac{1}{2}  ({\idm} + U_4^T)P & -(U_2+U_3)\\
\EA\right)
\quad \mbox{where} \quad U=(U_1,U_2,U_3,U_4),
\EEQ
will be an $\epsilon/2$ approximation of the objective function in (\ref{eq:dual-binary}). Whenever $\|U\|_F\leq 1$, we must have
\[
\left\|\left(\BA{cc}
-P^T U_1 P &  P^T U_4/2 \\
U_4^T P/2 & -(U_2+U_3)\\
\EA\right)\right\|_2^2 \leq \|P^TU_1P\|_2^2+ \|U_2+U_3\|_2^2 + \|P^TU_4\|_2^2 \leq 4,
\]
hence, following \cite[\S4]{Nest04a}, the gradient of $f(U)$ is Lipschitz continuous with respect to the Frobenius norm, with Lipschitz constant given by
\[
L=\frac{8 \log (n+p)}{\epsilon},
\]
We then define the compact, convex set $Q$ as
\[
Q\equiv\left\{(U_1,U_2,U_3,U_4) \in\symm_n^3:\|U_1\|_\infty + k^2 \|U_2\|_\infty + \|U_3\|_1 + k \|U_4\|_\infty \leq \bar \alpha\right\},
\]
and define a prox function $d(U)$ over $Q$ as $d(U)=\|U\|_F^2/2$, which is strongly convex with constant $\sigma=1$ w.r.t. the Frobenius norm. Starting from $U_0=0$, the algorithm in \cite{Nest03} for solving
\[\BA{ll}
\mbox{maximize} & f(U)\\
\mbox{subject to} & U \in Q, 
\EA\]
where $f(U)$ is defined in (\ref{eq:f-def}), proceeds as follows.

\vskip .1in
\noindent {\bf Repeat:}
\begin{enumerate} \itemsep 0ex
\item Compute $f(U_j)$ and $\nabla f(U_j)$
\item Find $Y_j = \arg\min_{Y \in Q} \: \langle \nabla f(U_j) , Y \rangle + \frac{1}{2} L \|U_i-Y\|_F^2$
\item Find $W_j = \arg\min_{W\in Q} \left\{ \frac{L d(W)}{\sigma}  + \sum_{j=0}^i \frac{j+1}{2}(f(U_j)+\langle \nabla f(U_j),W-U_j \rangle ) \right\}$
\item Set $U_{j+1} = \frac{2}{j+3} W_j + \frac{j+1}{j+3} Y_j$
\end{enumerate}
\noindent {\bf Until} gap $\leq\epsilon$.
\vskip .1in

Step one above computes the (smooth) function value and gradient. The second step computes the \emph{gradient mapping}, which matches the gradient step for unconstrained problems (see \cite[p.86]{Nest03a}). Step three and four update an \emph{estimate sequence} see \cite[p.72]{Nest03a} of $f$ whose minimum can be computed explicitly and gives an increasingly tight upper bound on the minimum of $f$. We now present these steps in detail for our problem.

\paragraph{Step 1}
The most expensive step in the algorithm is the first, the computation of $f$ and its gradient. This amounts to computing the matrix exponential in (\ref{eq:fmu}) at a cost of $O(n^3)$ (see \cite{Mole03} for details).
\paragraph{Step 2}
This step involves solving a problem of the form
\[
\argmin_{Y \in Q} \: \langle \nabla f(U) , Y \rangle + \frac{1}{2} L \|U-Y\|_F^2 ,
\]
where $U$ is given.  The above problem can be reduced to an Euclidean projection on $Q$
\begin{equation}\label{eq:euclo-prox}
\argmin_{\|Y\|\in Q} \: \|Y - V\|_F,
\end{equation}
where $V = U + L^{-1}\nabla f_\mu(U)$ is given. According to Lemma \ref{lem:proj}, this can be solved  $O(n\log n\log_2 (1/\epsilon))$ opearations.

\paragraph{Step 3} The third step involves solving an Euclidean projection problem similar to (\ref{eq:euclo-prox}), with $V$ defined here by:
\[
V = \frac{\sigma}{L} \sum_{j=0}^i \frac{j+1}{2} \nabla f_\mu(U_j).
\]

\paragraph{Stopping criterion}  We stop the algorithm when the duality gap is smaller than the target precision $\epsilon$. The dual of the binary optimization problem (\ref{eq:dual-binary}) can be written
\BEQ\label{eq:bin-dual}
\BA{ll}
\mbox{minimize} & \bar \alpha \max\{\|PG_{11}P^T\|_1,\frac{\|G_{22}\|_1}{k^2},\|G_{22}\|_\infty,\frac{\|PG_{12}\|_1}{k}\}-\Tr(PG_{12})\\
\mbox{subject to} & \Tr(G)=1,~G\succeq 0,\\
\EA\EEQ
in the block matrix variable $G\in\symm_{n+p}$ with blocks $G_{ij}$, $i,j=1,2$. Since the gradient $\nabla f(U)$ produces a dual feasible point by construction, we can use it to compute a dual objective value and bound the duality gap at the current point $U$.

\paragraph{Complexity}  According to \cite{Nest04a}, the total worst-case complexity to solve (\ref{eq:dual-binary}) with absolute accuracy less than $\epsilon$ is then given by
\[
O\left( \frac{n^4\sqrt{\log n}}{\epsilon}\right)
\]
Each iteration of the algorithm requires computing a matrix exponential at a cost of $O(n^3)$ and the algorithm requires $O(n\sqrt{\log n }/\epsilon)$ iterations to reach a target precision of $\epsilon>0$. Note that while this smooth optimization method can be used to produce reasonable complexity bounds for checking if the optimal value of (\ref{eq:dual-binary}) is positive, i.e. if $\alpha_k \leq \bar \alpha$, in practice the algorithm is relatively slow and we mostly use interior point solvers on smaller problems to conduct experiments in the next section.

\section{Numerical Results}
\label{s:numres} In this section, we illustrate the numerical performance of the semidefinite relaxation detailed in section~\ref{s:relax}.

\subsection{Illustration}
We test the semidefinite relaxation in (\ref{eq:dual-ker}) on a sample of ten random Gaussian matrices $A\in\reals^{p \times n}$ with $A_{ij}\sim\mathcal{N}(0,1/\sqrt{p})$, $n=30$ and $p=22$. For each of these matrices, we solve problem (\ref{eq:dual-ker}) for $k=2,\ldots,5$ to produce upper bounds on $\alpha_k$, hence on $C_k$ in (\ref{def:nullspace}), with $\alpha_k=1-1/C_k$. From \cite{Dono01}, we know that if $\alpha_k<1$ then we can bound the decoding error in (\ref{eq:decoding-error}), and if  $\alpha_k<1/2$ then the original signal can be recovered exactly by solving a linear program. We also plot the randomized values for $y^Tx$ with $k=1$ together with the semidefinite relaxation bound. 

\begin{figure}[!h]
\begin{center}
\begin{tabular}{cc}
\psfrag{k}[t][b]{Cardinality}
\psfrag{alphak}[b][t]{Bounds on $\alpha_k$}
\psfrag{l1}{$\ell_1$ recovery}
\psfrag{l0}[c][r]{$\ell_0$ recovery}
\includegraphics[width=.49\textwidth]{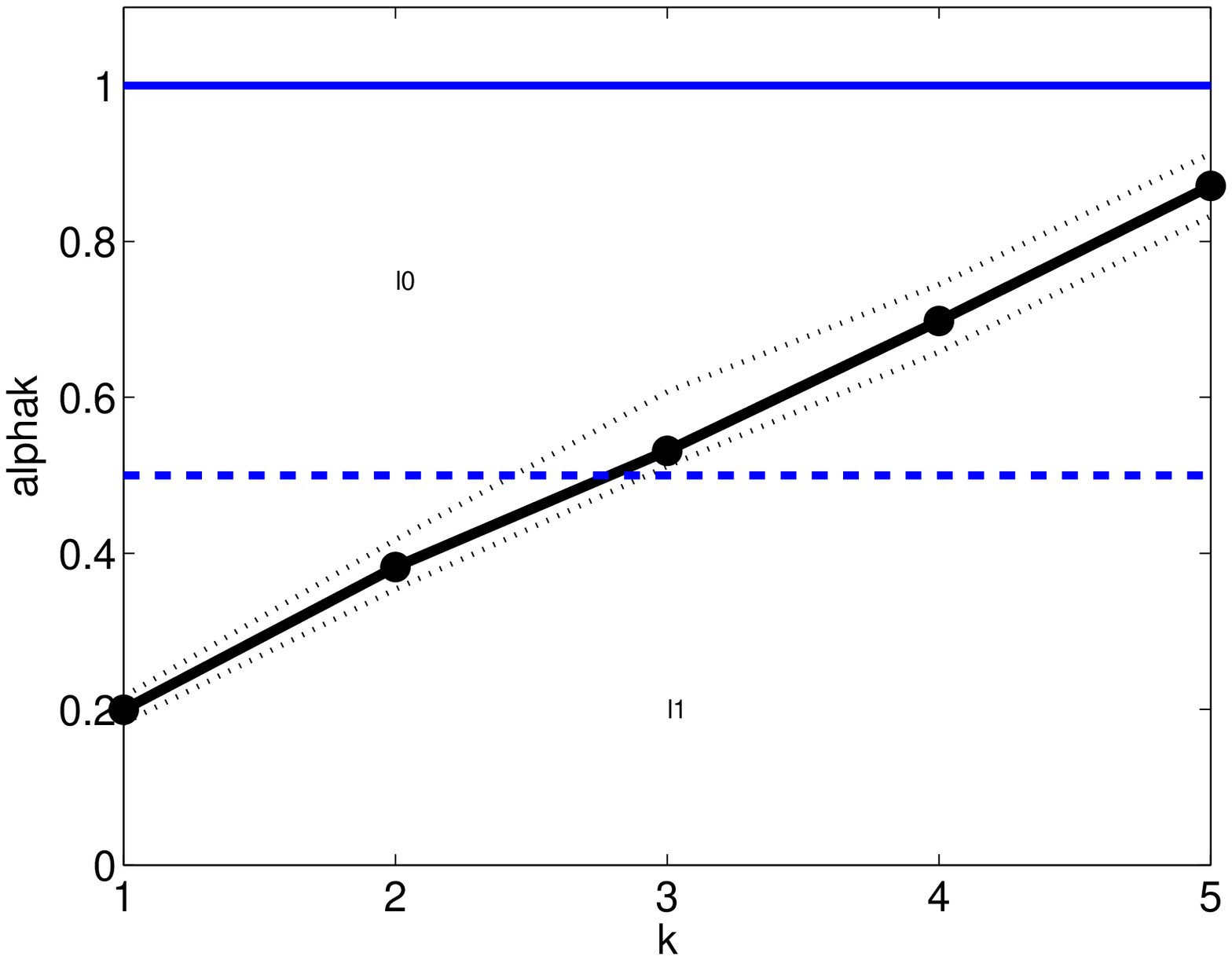} &
\psfrag{alphak}[t][b]{$\alpha_1$}
\psfrag{numoccur}[b][t]{Number of samples}
\psfrag{sdp}[t][b]{SDP}
\includegraphics[width=.49\textwidth]{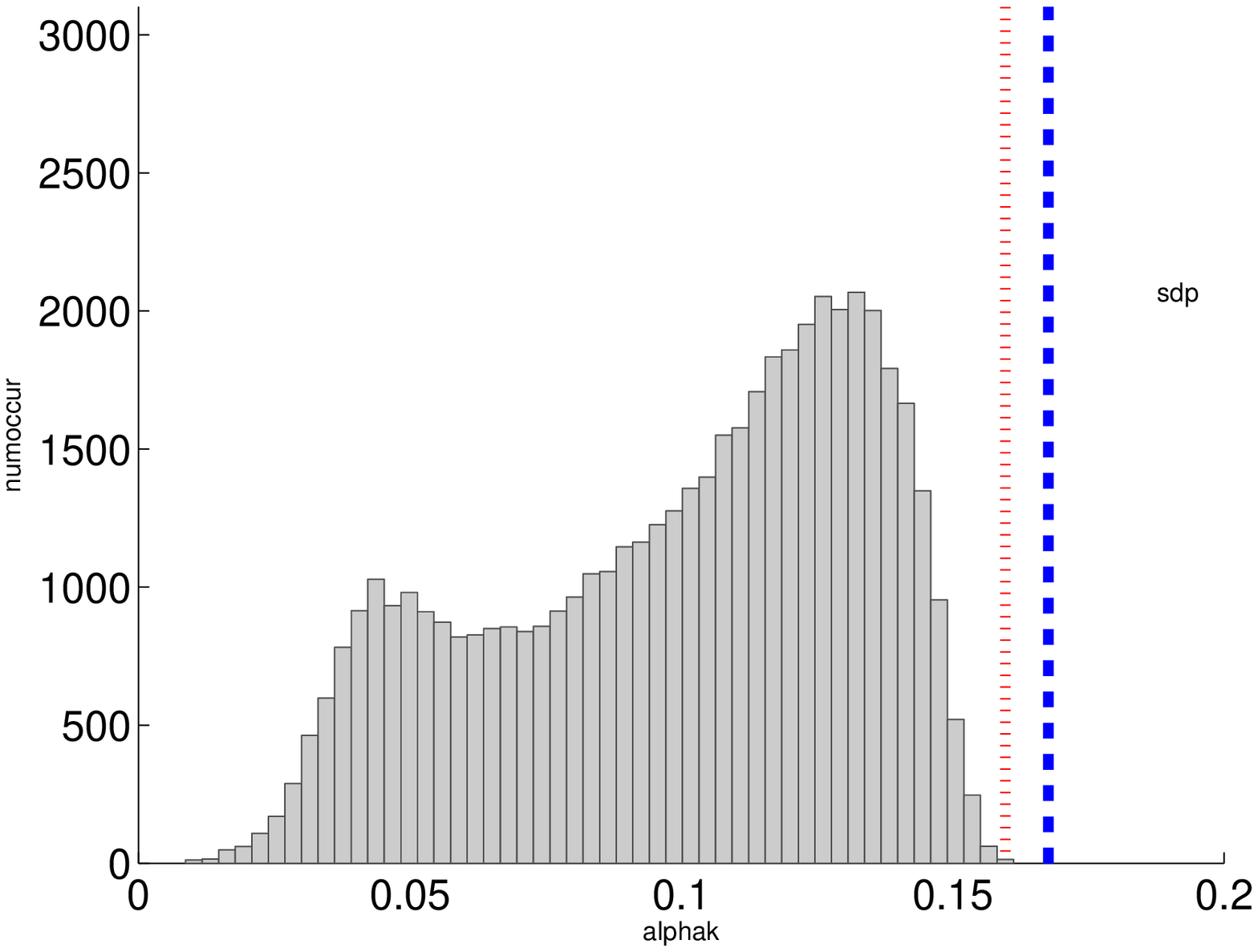}\\
\end{tabular}
\end{center}
\caption{{\bf Bounds on $\alpha_k$.} \emph{Left:} Upper bounds on $\alpha_k$ obtained by solving (\ref{eq:dual-ker}) for various values of $k$. Median bound over ten samples (solid line), dotted lines at pointwise minimum and maximum. \emph{Right:} Lower bound on $\alpha_1$ obtained by randomization (red dotted line) compared with semidefinite relaxation bound (SDP dashed line).\label{fig:alpha-bounds}}
\end{figure}

\begin{figure}[!h]
\begin{center}
\begin{tabular}{cc}
\psfrag{k}[t][b]{Cardinality}
\psfrag{prob}[b][t]{Probability of recovery}
\includegraphics[width=.49\textwidth]{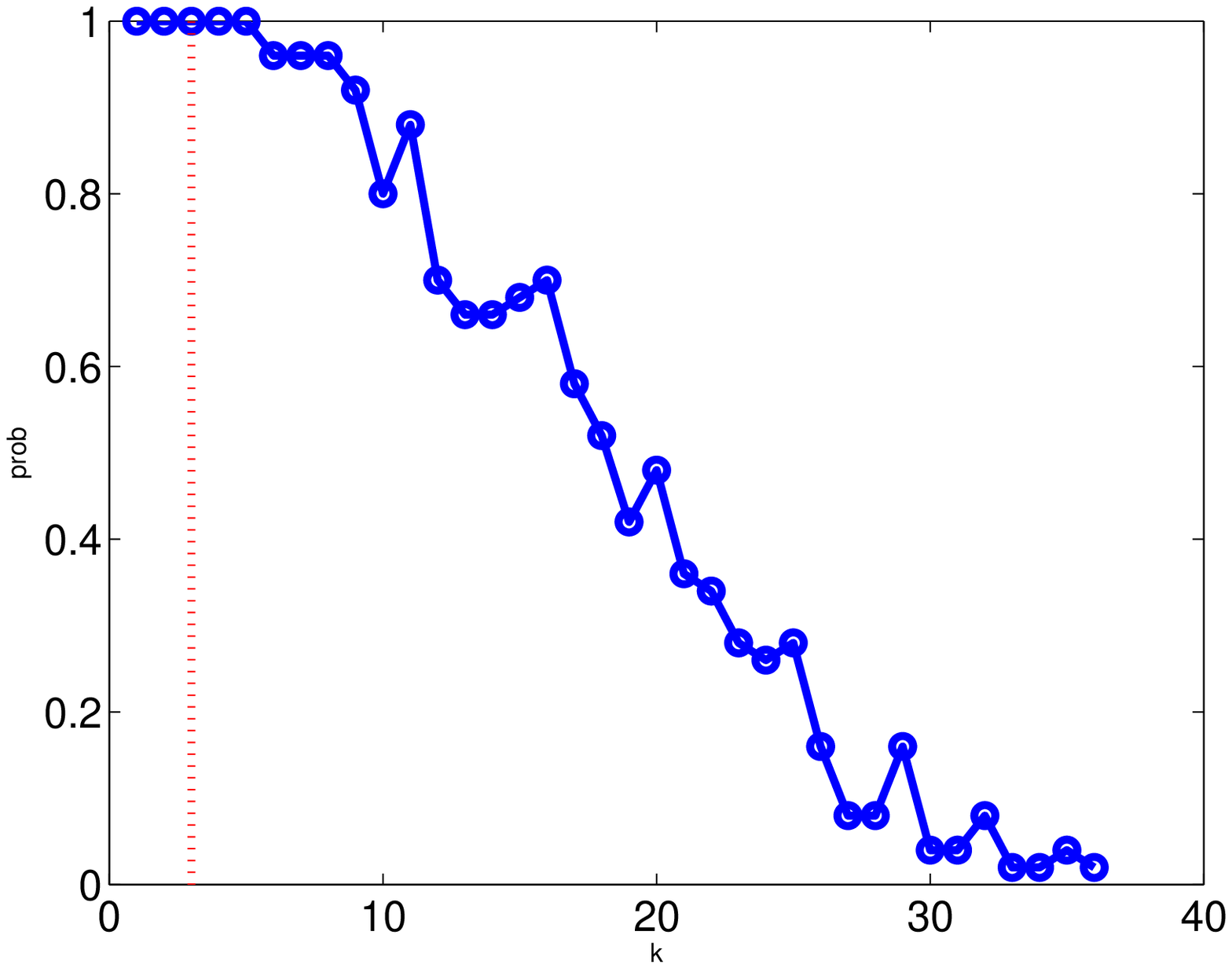} &
\psfrag{k}[t][b]{Cardinality}
\psfrag{l1}[b][t]{Mean $\ell_1$ error}
\includegraphics[width=.49\textwidth]{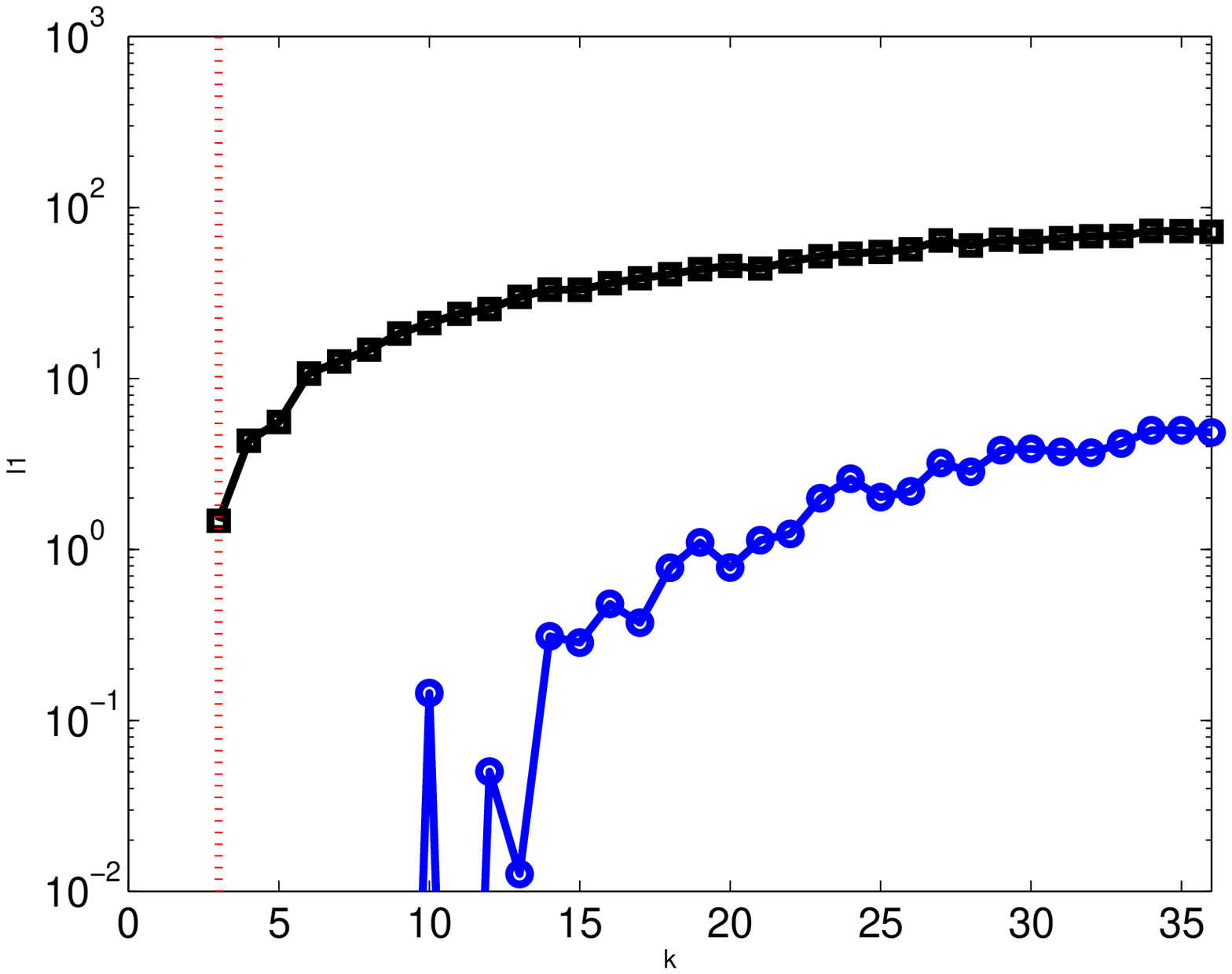}\\
\end{tabular}
\end{center}
\caption{{\bf Sparse Recovery.} \emph{Left:} Empirical probability of recovering the original sparse signal using the LP decoder in~(\ref{eq:delta-one}). The dashed line is at the strong recovery threshold. \emph{Right:} Empirical mean $\ell_1$ recovery error $\|x-x_0\|_1$ using the LP decoder (circles) compared with the bound induced by Theorem~\ref{th:lp-recov} (squares). \label{fig:mse-bounds}}
\end{figure}

Next, in Figure \ref{fig:mse-bounds}, we use a Gaussian matrix $A\in\reals^{p \times n}$ with $A_{ij}\sim\mathcal{N}(0,1/\sqrt{p})$, $n=36$ and $p=27$ and, for each $k$, we sample fifty information vectors $v=Ax_0$ where $x_0$ is uniformly distributed and has cardinality~$k$. On the left, we plot the probability of recovering the original sparse signal $x_0$ using the linear programming decoder in (\ref{eq:delta-one}). On the right, we plot the mean $\ell_1$ recovery error $\|x-x_0\|_1$ using the linear programming decoder in (\ref{eq:delta-one}) and compare it with the bound induced by Theorem~\ref{th:lp-recov}. 

\subsection{Performance on compressed sensing matrices}
In tables \ref{tab:fourier}, \ref{tab:gauss} and \ref{tab:bernou}, we compare the performance of the linear programming relaxation bound on~$\alpha_k$ derived in \cite{Judi08} with that of the semidefinite programming bound detailed in Section~\ref{s:relax}. We test these bounds for various matrix shape ratios $\rho=m/n$, target cardinalities $k$ on matrices with Fourier, Bernoulli or Gaussian coefficients using SDPT3 by \cite{Toh96} to solve problem (\ref{eq:dual-ker}). We show median bounds computed over ten sample matrices for each type, hence test a total of 600 different matrices. We compare these relaxation bounds with the upper bounds produced by sequential convex optimization as in \citet[\S4.1]{Judi08}. In the Gaussian case, we also compare these relaxation bounds with the asymptotic thresholds on strong and weak (high probability) recovery discussed in \cite{Dono08}. The semidefinite bounds on $\alpha_k$ always match with the LP bounds in \cite{Judi08} when $k=1$ (both are tight), and are often smaller than LP bounds whenever $k$ is greater than 1 on Gaussian or Bernoulli matrices. The semidefinite upper bound on $\alpha_k$ was smaller than the LP one in 563 out of the 600 matrices sampled here, with the difference ranging from 4e-2 to -9e-4. Of course, this semidefinite relaxation is significantly more expensive than the LP based one and that these experiments thus had to be performed on very small matrices. 

\begin{table}[h]
\small{
\begin{center}
\begin{tabular}{r|c|c|c|c|c|c|c}
Relaxation & $\rho$ & $\alpha_1$ & $\alpha_2$ & $\alpha_3$ & $\alpha_4$ & $\alpha_5$ & Upper bound \\
\hline 
LP & 0.5 & \bf 0.21 & \bf 0.38 & 0.57 & 0.82 & 0.98 & 2 \\
SDP & 0.5 & \bf 0.21 & \bf 0.38 & 0.57 & 0.82 & 0.98 & 2\\
SDP low. & 0.5 & 0.05 & 0.10 & 0.16 & 0.24 & 0.32 & 2\\
\hline
LP & 0.6 & \bf 0.16 & \bf 0.31 & \bf 0.46 & 0.61 & 0.82 & 3\\
SDP & 0.6 & \bf 0.16 & \bf 0.31 & \bf 0.46 & 0.61 & 0.82 & 3\\
SDP low. & 0.6 & 0.04 & 0.09 & 0.15 & 0.20 & 0.31 & 3\\
\hline
LP & 0.7 & \bf 0.12 & \bf 0.25 & \bf 0.39 & 0.50 & 0.62 & 4\\
SDP & 0.7 & \bf 0.12 & \bf 0.25 & \bf 0.39 & 0.50 & 0.62 & 4\\
SDP low. & 0.7 & 0.04 & 0.09 & 0.14 & 0.18 & 0.22 & 4\\
\hline
LP & 0.8 & \bf 0.10 & \bf 0.20 & \bf 0.30 & \bf 0.38 & \bf 0.48 & 6\\
SDP & 0.8 & \bf 0.10 & \bf 0.20 & \bf 0.30 & \bf 0.38 & \bf 0.48 & 6\\
SDP low. & 0.8 & 0.04 & 0.07 & 0.13 & 0.17 & 0.23  & 6\\
\end{tabular}
\end{center}
}
\caption{Given ten sample {\em Fourier} matrices of leading dimension $n=40$, we list median upper bounds on the values of $\alpha_k$ for various cardinalities $k$ and matrix shape ratios $\rho$, computed using the linear programming (LP) relaxation in \cite{Judi08} and the semidefinite relaxation (SDP) detailed in this paper. We also list the upper bound on strong recovery computed using sequential convex optimization and the lower bound on~$\alpha_k$ obtained by randomization using the SDP solution (SDP low.). Values of $\alpha_k$ below 1/2, for which strong recovery is certified, are highlighted in bold. \label{tab:fourier}}
\end{table}

\begin{table}[p]
\small{
\begin{center}
\begin{tabular}{r|c|c|c|c|c|c|c|c}
Relaxation & $\rho$ & $\alpha_1$ & $\alpha_2$ & $\alpha_3$ & $\alpha_4$ & $\alpha_5$ & Strong $k$ & Weak $k$ \\
\hline 
LP & 0.5 & \bf 0.27 & \bf 0.49 & 0.67 & 0.83 & 0.97 & 2 & 11\\
SDP & 0.5 & \bf 0.27 & \bf 0.49 & 0.65 & 0.81 & 0.94 & 2 & 11 \\
SDP low. & 0.5 & 0.27 & 0.31 & 0.33 & 0.32 & 0.35 & 2 & 11\\
\hline 
LP & 0.6 & \bf 0.22 & \bf 0.41 & 0.57 & 0.72 & 0.84 & 2 & 12\\
SDP & 0.6 & \bf 0.22 & \bf 0.41 & 0.56 & 0.70 & 0.82 & 2 & 12\\
SDP low. & 0.6 & 0.22 & 0.29 & 0.31 & 0.32 & 0.36 & 2 & 12\\
\hline 
LP & 0.7 & \bf 0.20 & \bf 0.34 & \bf 0.47 & 0.60 & 0.71 & 3 & 14\\
SDP & 0.7 & \bf 0.20 & \bf 0.34 & \bf 0.46 & 0.59 & 0.70 & 3 & 14\\
SDP low. & 0.7 & 0.20 & 0.27 & 0.31 & 0.35 & 0.38 & 3 & 14\\
\hline 
LP & 0.8 & \bf 0.15 & \bf 0.26 & \bf 0.37 & \bf 0.48 & 0.58 & 3 & 16\\
SDP & 0.8 & \bf 0.15 & \bf 0.26 & \bf 0.37 & \bf 0.48 & 0.58 & 3 & 16\\
SDP low. & 0.8 & 0.15 & 0.23 & 0.28 & 0.33 & 0.38 & 3 & 16\\
\end{tabular}
\end{center}
}
\caption{Given ten sample {\em Gaussian} matrices of leading dimension $n=40$, we list median upper bounds on the values of $\alpha_k$ for various cardinalities $k$ and matrix shape ratios $\rho$, computed using the linear programming (LP) relaxation in \cite{Judi08} and the semidefinite relaxation (SDP) detailed in this paper. We also list the asymptotic upper bound on both strong and weak recovery computed in \cite{Dono08} and the lower bound on~$\alpha_k$ obtained by randomization using the SDP solution (SDP low.). Values of $\alpha_k$ below 1/2, for which strong recovery is certified, are highlighted in bold. \label{tab:gauss}}
\end{table}

\begin{table}[p]
\small{
\begin{center}
\begin{tabular}{r|c|c|c|c|c|c|c}
Relaxation & $\rho$ & $\alpha_1$ & $\alpha_2$ & $\alpha_3$ & $\alpha_4$ & $\alpha_5$ & Upper bound\\
\hline 
LP & 0.5 & \bf 0.25 & \bf 0.45 & 0.64 & 0.82 & 0.97 & 2\\
SDP & 0.5 & \bf 0.25 & \bf 0.45 & 0.63 & 0.80 & 0.94 & 2\\
SDP low. & 0.5 & 0.25 & 0.28 & 0.29 & 0.29 & 0.34 & 2\\
\hline 
LP & 0.6 & \bf 0.21 & \bf 0.38 & 0.55 & 0.69 & 0.83 & 3\\
SDP & 0.6 & \bf 0.21 & \bf 0.38 & 0.54 & 0.68 & 0.81 & 3\\
SDP low. & 0.6 & 0.21 & 0.26 & 0.29 & 0.33 & 0.34 & 3\\
\hline 
LP & 0.7 & \bf 0.17 & \bf 0.32 & \bf 0.46 & 0.58 & 0.70 & 4\\
SDP & 0.7 & \bf 0.17 & \bf 0.32 & \bf 0.46 & 0.58 & 0.69 & 4\\
SDP low. & 0.7 & 0.17 & 0.24 & 0.29 & 0.33 & 0.37 & 4\\
\hline 
LP & 0.8 & \bf 0.14 & \bf 0.26 & \bf 0.38 & \bf 0.47 & 0.57 & 5\\
SDP & 0.8 & \bf 0.14 & \bf 0.26 & \bf 0.37 & \bf 0.47 & 0.57 & 5\\
SDP low. & 0.8 & 0.14 & 0.21 & 0.27 & 0.33 & 0.38 & 5\\
\end{tabular}
\end{center}
}
\caption{Given ten sample {\em Bernoulli} matrices of leading dimension $n=40$, we list median upper bounds on the values of $\alpha_k$ for various cardinalities $k$ and matrix shape ratios $\rho$, computed using the linear programming (LP) relaxation in \cite{Judi08} and the semidefinite relaxation (SDP) detailed in this paper. We also list the upper bound on strong recovery computed using sequential convex optimization and the lower bound on~$\alpha_k$ obtained by randomization using the SDP solution (SDP low.). Values of $\alpha_k$ below 1/2, for which strong recovery is certified, are highlighted in bold. \label{tab:bernou}}
\end{table}

\subsection{Tightness}
Section \ref{s:tight} shows that the tightness of the semidefinite relaxation is explicitly controlled by the following quantity
\[
\mu=g(X,\delta)h(Y,n,k,\delta),
\] 
where $g$ and $h$ are defined in (\ref{eq:defg}) and (\ref{eq:defh}) respectively. In Figure \ref{fig:gamma-hist}, we plot the histogram of values of $\mu$ for all 600 sample matrices computed above, and plot the same histogram on a subset of these results where the target cardinality $k$ was set to 1. We observe that while the relaxation performed quite well on most of these examples, the randomization bound on performance often gets very large whenever $k>1$. This can probably be explained by the fact that we only control the mean in Lemma \ref{lem:dev-xy}, not the quantile. We also notice that $\mu$ is highly concentrated when $k=1$ on Gaussian and Bernoulli matrices (where the results in Tables \ref{tab:gauss} and \ref{tab:bernou} are tight), while the performance is markedly worse for Fourier matrices. 

Finally, Tables \ref{tab:gauss} and \ref{tab:bernou} show that lower bounds on $\alpha_1$ obtained by randomization for Gaussian are always tight (the solution of the SDP was very close to rank one), while performance on higher values of $k$ and Fourier matrices is much worse. On 6 of these experiments however, the SDP randomization lower bound was higher than 1/2, which proved that $\alpha_5>1/2$, hence that the matrix did not satisfy the nullspace property at order 5.

\begin{figure}[!h]
\begin{center}
\begin{tabular}{cc}
\psfrag{XY12}[t][b]{$\mu$}
\psfrag{numoccur}[b][t]{Number of samples}
\includegraphics[width=.49\textwidth]{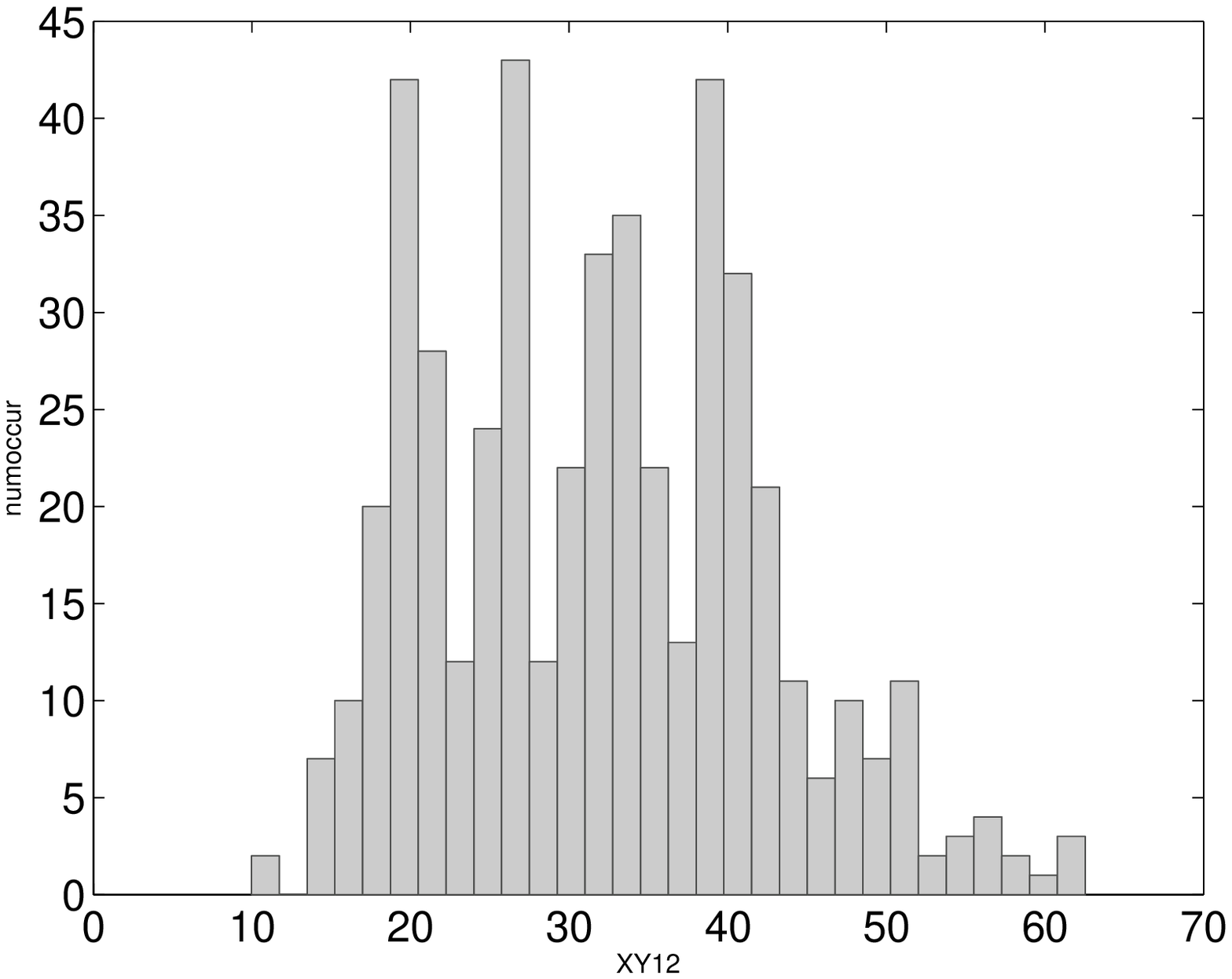} &
\psfrag{1XY12}[t][b]{$\mu$}
\psfrag{numoccur}[b][t]{Number of samples}
\includegraphics[width=.49\textwidth]{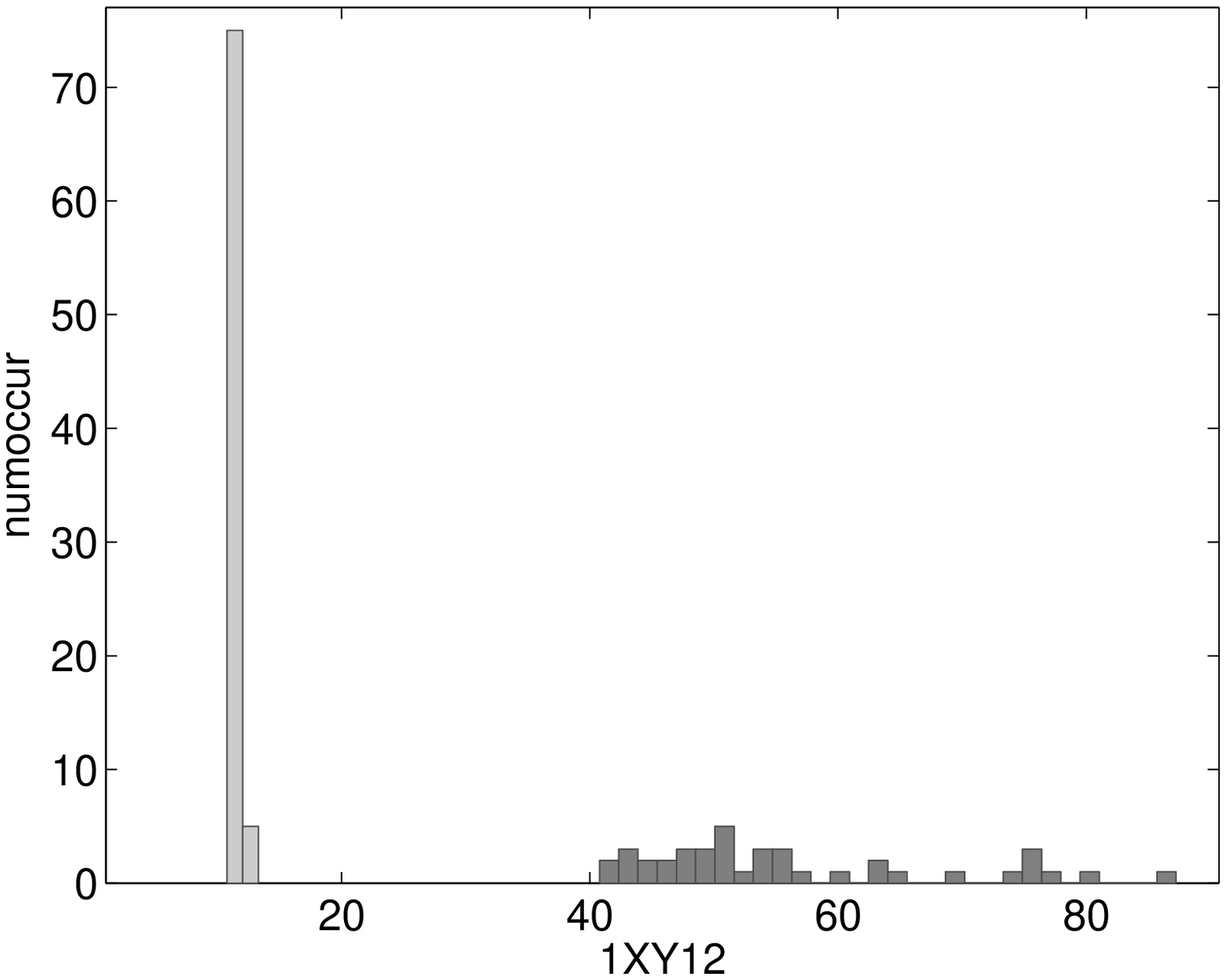}\\
\end{tabular}
\end{center}
\caption{{\bf Tightness.} \emph{Left:} Histogram of $\mu=g(X,\delta)h(Y,n,k,\delta)$ defined in (\ref{eq:defg}) and (\ref{eq:defh}), computed for all sample solution matrices in the experiments above when $k>1$.  \emph{Right:} Idem using only examples where the target cardinality is $k=1$, for Gaussian and Bernoulli matrices (light grey) or Fourier matrices (dark grey).\label{fig:gamma-hist}}
\end{figure}

\subsection{Numerical complexity}
We implemented the algorithm of Section (\ref{s:algos}) in MATLAB and tested it on random matrices. While the code handles matrices with $n=500$, it is still considerably slower than similar first-order algorithms applied to sparse PCA problems for example (see \cite{dAsp04a}). A possible explanation for this gap in performance is perhaps that the DSPCA semidefinite relaxation is always tight (in practice at least) hence iterates near the solution tend to be very close to rank one. This is not the case here as the matrix in (\ref{eq:max-max-relax}) is very rarely rank one and the number of significant eigenvalues has a direct impact on actual convergence speed. To illustrate this point, Figure \ref{fig:cpu} shows a Scree plot of the optimal solution to~(\ref{eq:max-max-relax}) for a small Gaussian matrix (obtained by IP methods with a target precision of~$10^{-8}$), while Table \ref{tab:cpu} shows, as a benchmark, total CPU time for proving that $\alpha_1<1/2$ on Gaussian matrices, for various values of $n$. We set the accuracy $1e-2$ and stop the code whenever positive objective values are reached. Unfortunately, performance for larger values of $k$ is typically much worse (which is why we used IP methods to run most experiments in this section) and in many cases, convergence is hard to track as the dual objective values computed using the gradient in (\ref{eq:bin-dual}) produces a relatively coarse gap bounds as illustrated in Figure \ref{fig:cpu} for a small Gaussian matrix.

\begin{table}
\begin{center}
\begin{tabular}{|r|c|c|c|c|}
\hline
$n$ & 50 & 100 & 200 & 500\\
\hline
CPU time &  00 h  01 m &  00 h 10 m &  01 h 38 m & 37 h 22 m\\
\hline
\end{tabular}
\end{center}
\caption{CPU time to show $\alpha_1<1/2$, using the algorithm of Section \ref{s:algos} on Gaussian matrices with shape ratio $\rho=.7$ for various values of $n$.\label{tab:cpu}}
\end{table}

\begin{figure}[!h]
\begin{center}
\begin{tabular}{cc}
\psfrag{iter}[t][b]{Iterations}
\psfrag{alphak}[b][t]{$\alpha_k$}
\includegraphics[width=.49\textwidth]{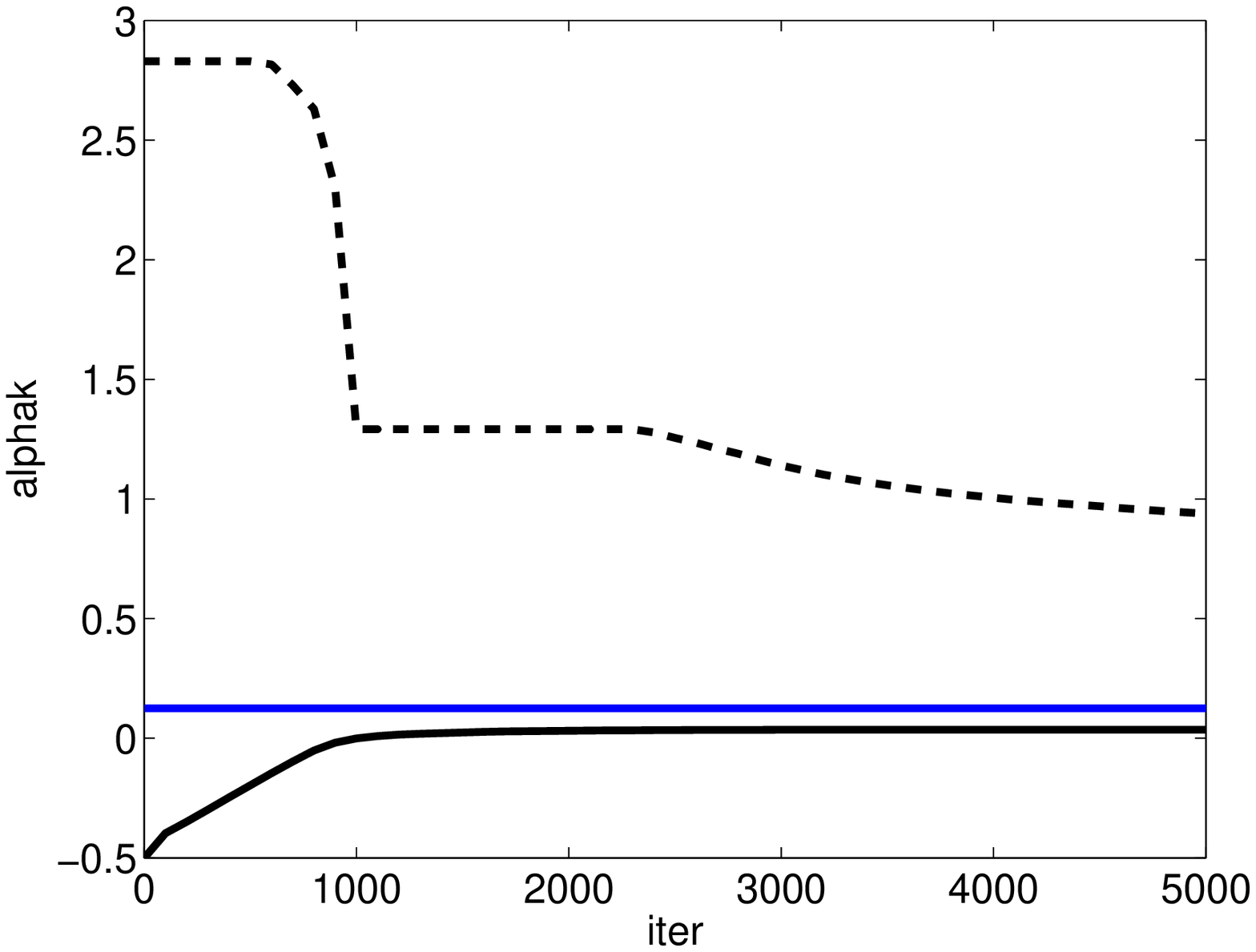} &
\psfrag{eigs}[t][b]{Eigenvalues}
\includegraphics[width=.49\textwidth]{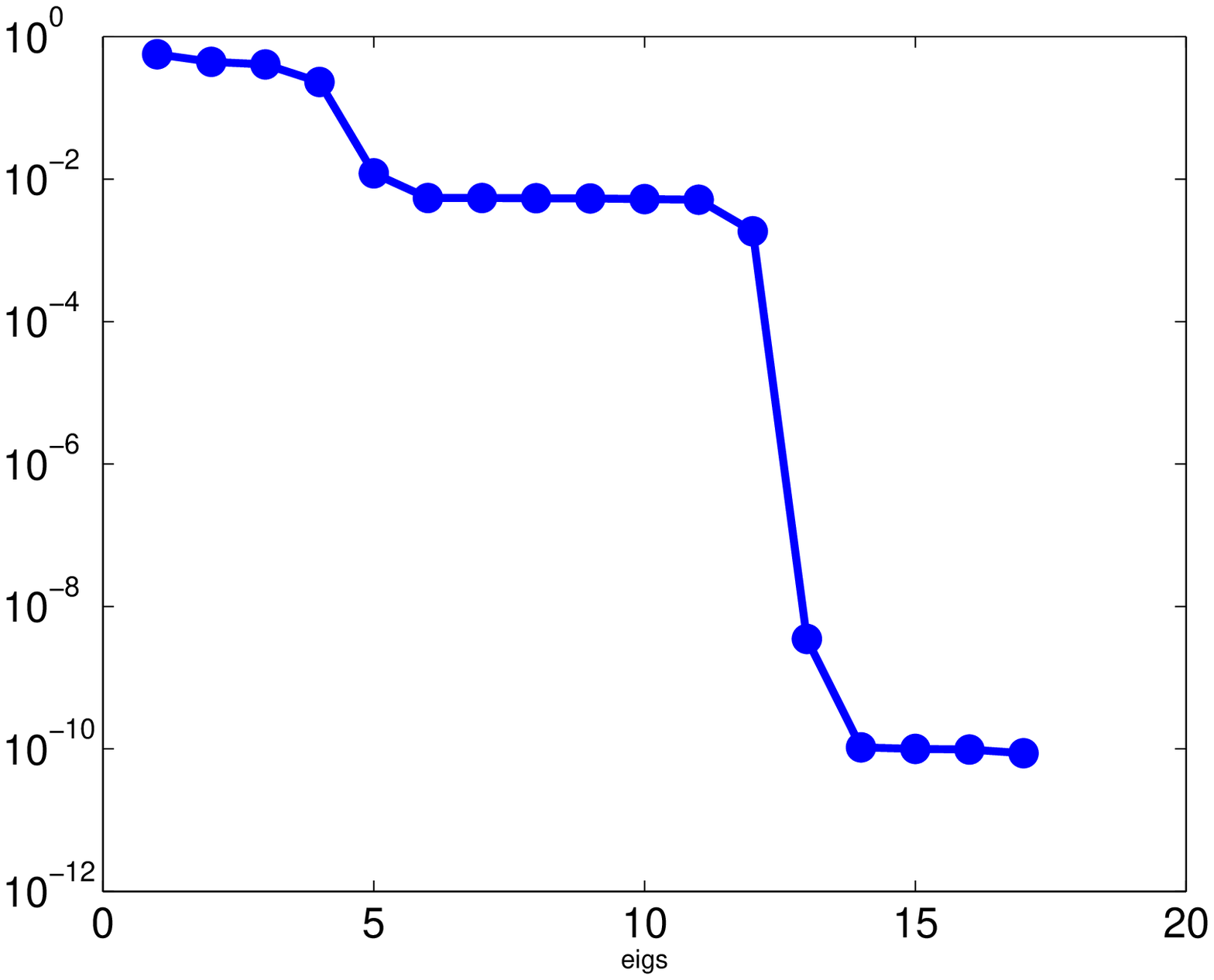}\\
\end{tabular}
\end{center}
\caption{{\bf Complexity.} \emph{Left:} Primal and dual bounds on the optimal solution (computed using interior point methods) using the algorithm of Section   \ref{s:algos} on a small Gaussian matrix. \emph{Right:} Scree plot of the optimal solution to (\ref{eq:max-max-relax}) for a small Gaussian matrix (obtained by interior point methods with a target precision of~$10^{-8}$). \label{fig:cpu}}
\end{figure}

\section{Conclusion \& Directions for Further Research}
We have detailed a semidefinite relaxation for the problem of testing if a matrix satisfies the nullspace property defined in \cite{Dono01} or \cite{Cohe06}. This relaxation is tight for $k=1$ and matches (numerically) the linear programming relaxation in \cite{Judi08}. It is often slightly tighter (again numerically) for larger values of $k$. We can also remark that the matrix $A$ only appears in the relaxation (\ref{eq:sdp-kernel}) in ``kernel'' format $A^TA$, where the constraints are linear in the kernel matrix $A^TA$. This means that this relaxation might allow sparse experiment design problems to be solved, while maintaining convexity.

Of course, these small scale experiments do not really shed light on the actual performance of both relaxations on larger, more realistic problems. In particular, applications in imaging and signal processing would require solving problems where both $n$ and $k$ are several orders of magnitude larger than the values considered in this paper or in \cite{Judi08} and the question of finding tractable relaxations or algorithms that can handle such problem sizes remains open. Finally, the three different tractable tests for sparse recovery conditions, derived in \cite{dAsp08b}, \cite{Judi08} and this paper, are {\em all} limited to showing recovery at the (suboptimal) rate $k=O(\sqrt{m})$. Finding tractable test for sparse recovery at cardinalities $k$ closer to the optimal rate $O(m)$ also remains an open problem.

\section*{Acknowledgements}
The authors are grateful to Arkadi Nemirovski (who suggested in particular the columnwise redundant constraints in (\ref{eq:max-max-relax-col}) and the performance bounds) and Anatoli Juditsky for very helpful comments and suggestions. Would like to thank Ingrid Daubechies for first attracting our attention to the nullspace property. We are also grateful to two anonymous referees for numerous comments and suggestions. We thank Jared Tanner for forwarding us his numerical results. Finally, we acknowledge support from NSF grant DMS-0625352, NSF CDI grant SES-0835550, a NSF CAREER award, a Peek junior faculty fellowship and a Howard B. Wentz Jr. junior faculty award.

\small{\bibliographystyle{plainnat}
\bibliography{MainPerso}}

\end{document}

%% file: defs.tex
\newcommand{\BEAS}{\begin{eqnarray*}}
\newcommand{\EEAS}{\end{eqnarray*}}
\newcommand{\BEA}{\begin{eqnarray}}
\newcommand{\EEA}{\end{eqnarray}}
\newcommand{\BEQ}{\begin{equation}}
\newcommand{\EEQ}{\end{equation}}
\newcommand{\BIT}{\begin{itemize}}
\newcommand{\EIT}{\end{itemize}}
\newcommand{\BNUM}{\begin{enumerate}}
\newcommand{\ENUM}{\end{enumerate}}

\newcommand{\BA}{\begin{array}}
\newcommand{\EA}{\end{array}}

\newcommand{\ones}{\mathbf 1}

\newcommand{\reals}{{\mbox{\bf R}}}

\newcommand{\symm}{{\mbox{\bf S}}}  

\newcommand{\Rank}{\mathop{\bf Rank}}
\newcommand{\Card}{\mathop{\bf Card}}
\newcommand{\Tr}{\mathop{\bf Tr}}
\newcommand{\diag}{\mathop{\bf diag}}
\newcommand{\lambdamax}{{\lambda_{\rm max}}}
\newcommand{\lambdamin}{\lambda_{\rm min}}
\newcommand{\idm}{\bf I}

\newcommand{\Expect}{\textstyle\mathop{\bf E{}}}

\newcommand{\QED}{~~\rule[-1pt]{6pt}{6pt}}

\newcommand{\argmin}{\mathop{\rm argmin}}

\newtheorem{theorem}{Theorem}

\newtheorem{remark}[theorem]{Remark}

\newtheorem{definition}[theorem]{Definition}

\newcounter{exno}

\newenvironment{proof}{\textbf{Proof.}}{\QED\bigskip}

{\begin{quote}}{\end{quote}}

\makeatletter
\long\def\@makecaption#1#2{
   \vskip 9pt 
   \begin{small}
   \setbox\@tempboxa\hbox{{\bf #1:} #2}
   \ifdim \wd\@tempboxa > 5.5in
        \begin{center}
        \begin{minipage}[t]{5.5in}
        \addtolength{\baselineskip}{-0.95pt}
        {\bf #1:} #2 \par
        \addtolength{\baselineskip}{0.95pt}
        \end{minipage}
        \end{center}
   \else 
    \hbox to\hsize{\hfil\box\@tempboxa\hfil}  
   \fi
   \end{small}\par
}
\makeatother

\newcounter{oursection}

\newcounter{lecture}

%% file: NullSpacePropSDP10.bbl
\begin{thebibliography}{24}
\providecommand{\natexlab}[1]{#1}
\providecommand{\url}[1]{\texttt{#1}}
\expandafter\ifx\csname urlstyle\endcsname\relax
  \providecommand{\doi}[1]{doi: #1}\else
  \providecommand{\doi}{doi: \begingroup \urlstyle{rm}\Url}\fi

\bibitem[Affentranger and Schneider(1992)]{Affe92}
F.~Affentranger and R.~Schneider.
\newblock {Random projections of regular simplices}.
\newblock \emph{Discrete and Computational Geometry}, 7\penalty0 (1):\penalty0
  219--226, 1992.

\bibitem[Barvinok(2007)]{Barv07}
A.~Barvinok.
\newblock {Integration and optimization of multivariate polynomials by
  restriction onto a random subspace}.
\newblock \emph{Foundations of Computational Mathematics}, 7\penalty0
  (2):\penalty0 229--244, 2007.

\bibitem[Boyd et~al.(1994)Boyd, El~Ghaoui, Feron, and Balakrishnan]{Boyd94}
Stephen Boyd, Laurent El~Ghaoui, E.~Feron, and V.~Balakrishnan.
\newblock \emph{Linear Matrix Inequalities in System and Control Theory}.
\newblock SIAM, 1994.

\bibitem[Cand\`es and Tao(2005)]{Cand05}
E.~J. Cand\`es and T.~Tao.
\newblock Decoding by linear programming.
\newblock \emph{IEEE Transactions on Information Theory}, 51\penalty0
  (12):\penalty0 4203--4215, 2005.

\bibitem[Cand\`es and Tao(2006)]{Cand06}
E.J. Cand\`es and T.~Tao.
\newblock Near-optimal signal recovery from random projections: Universal
  encoding strategies?
\newblock \emph{IEEE Transactions on Information Theory}, 52\penalty0
  (12):\penalty0 5406--5425, 2006.

\bibitem[Cohen et~al.(2009)Cohen, Dahmen, and DeVore]{Cohe06}
A.~Cohen, W.~Dahmen, and R.~DeVore.
\newblock {Compressed sensing and best k-term approximation}.
\newblock \emph{Journal of the {AMS}}, 22\penalty0 (1):\penalty0 211--231,
  2009.

\bibitem[d'Aspremont et~al.(2007)d'Aspremont, El~Ghaoui, Jordan, and
  Lanckriet]{dAsp04a}
A.~d'Aspremont, L.~El~Ghaoui, M.I. Jordan, and G.~R.~G. Lanckriet.
\newblock A direct formulation for sparse {PCA} using semidefinite programming.
\newblock \emph{SIAM Review}, 49\penalty0 (3):\penalty0 434--448, 2007.

\bibitem[d'Aspremont et~al.(2008)d'Aspremont, Bach, and El~Ghaoui]{dAsp08b}
A.~d'Aspremont, F.~Bach, and L.~El~Ghaoui.
\newblock Optimal solutions for sparse principal component analysis.
\newblock \emph{Journal of Machine Learning Research}, 9:\penalty0 1269--1294,
  2008.

\bibitem[Donoho and Tanner(2005)]{Dono05}
D.~L. Donoho and J.~Tanner.
\newblock Sparse nonnegative solutions of underdetermined linear equations by
  linear programming.
\newblock \emph{Proc. of the National Academy of Sciences}, 102\penalty0
  (27):\penalty0 9446--9451, 2005.

\bibitem[Donoho and Huo(2001)]{Dono01}
D.L. Donoho and X.~Huo.
\newblock {Uncertainty principles and ideal atomic decomposition}.
\newblock \emph{IEEE Transactions on Information Theory}, 47\penalty0
  (7):\penalty0 2845--2862, 2001.

\bibitem[Donoho and Tanner(2008)]{Dono08}
D.L. Donoho and J.~Tanner.
\newblock {Counting the Faces of Randomly-Projected Hypercubes and Orthants,
  with Applications}.
\newblock \emph{Arxiv preprint arXiv:0807.3590}, 2008.

\bibitem[Goemans and Williamson(1995)]{Goem95}
M.X. Goemans and D.P. Williamson.
\newblock Improved approximation algorithms for maximum cut and satisfiability
  problems using semidefinite programming.
\newblock \emph{J. ACM}, 42:\penalty0 1115--1145, 1995.

\bibitem[Ibragimov et~al.(1976)Ibragimov, Sudakov, and Tsirelson]{Ibra76}
IA~Ibragimov, VN~Sudakov, and BS~Tsirelson.
\newblock {Norms of Gaussian sample functions}.
\newblock \emph{Proceedings of the Third Japan USSR Symposium on Probability
  theory, Lecture Notes in Math}, 550:\penalty0 20--41, 1976.

\bibitem[Juditsky and Nemirovski(2008)]{Judi08}
A.~Juditsky and A.S. Nemirovski.
\newblock On verifiable sufficient conditions for sparse signal recovery via
  $\ell_1$ minimization.
\newblock \emph{ArXiv:0809.2650}, 2008.

\bibitem[Lee and Bresler(2008)]{Lee08}
K.~Lee and Y.~Bresler.
\newblock {Computing performance guarantees for compressed sensing}.
\newblock In \emph{IEEE International Conference on Acoustics, Speech and
  Signal Processing, 2008. ICASSP 2008}, pages 5129--5132, 2008.

\bibitem[Massart(2007)]{Mass07}
P.~Massart.
\newblock Concentration inequalities and model selection.
\newblock \emph{Ecole d'Et\'e de Probabilit\'es de Saint-Flour XXXIII}, 2007.

\bibitem[Moler and Van~Loan(2003)]{Mole03}
C.~Moler and C.~Van~Loan.
\newblock Nineteen dubious ways to compute the exponential of a matrix,
  twenty-five years later.
\newblock \emph{{SIAM} Review}, 45\penalty0 (1):\penalty0 3--49, 2003.

\bibitem[Nesterov(2003)]{Nest03a}
Y.~Nesterov.
\newblock \emph{Introductory Lectures on Convex Optimization}.
\newblock Springer, 2003.

\bibitem[Nesterov(2005)]{Nest03}
Y.~Nesterov.
\newblock Smooth minimization of non-smooth functions.
\newblock \emph{Mathematical Programming}, 103\penalty0 (1):\penalty0 127--152,
  2005.

\bibitem[Nesterov(2007)]{Nest04a}
Y.~Nesterov.
\newblock Smoothing technique and its applications in semidefinite
  optimization.
\newblock \emph{Mathematical Programming}, 110\penalty0 (2):\penalty0 245--259,
  2007.

\bibitem[Sturm(1999)]{Stur99}
J.~Sturm.
\newblock Using {SEDUMI} 1.0x, a {MATLAB} toolbox for optimization over
  symmetric cones.
\newblock \emph{Optimization Methods and Software}, 11:\penalty0 625--653,
  1999.

\bibitem[Toh et~al.(1999)Toh, Todd, and Tutuncu]{Toh96}
K.~C. Toh, M.~J. Todd, and R.~H. Tutuncu.
\newblock {SDPT3} -- a {MATLAB} software package for semidefinite programming.
\newblock \emph{Optimization Methods and Software}, 11:\penalty0 545--581,
  1999.

\bibitem[Vershik and Sporyshev(1992)]{Vers92}
AM~Vershik and PV~Sporyshev.
\newblock Asymptotic behavior of the number of faces of random polyhedra and
  the neighborliness problem.
\newblock \emph{Selecta Math. Soviet}, 11\penalty0 (2):\penalty0 181--201,
  1992.

\bibitem[Zhang(2005)]{Zhan05}
Y.~Zhang.
\newblock A simple proof for recoverability of $\ell_1$-minimization: Go over
  or under.
\newblock \emph{Rice University CAAM Technical Report TR05-09}, 2005.

\end{thebibliography}
